\documentclass[11pt]{amsart}

\usepackage{amsmath,amsthm,amssymb,amsfonts}
\usepackage{color}
\usepackage{hyperref}

\usepackage{amscd}
\usepackage[latin1]{inputenc}
\DeclareMathAlphabet{\mathpzc}{OT1}{pzc}{m}{it}
\usepackage{bbm}

\numberwithin{equation}{section}

\swapnumbers
\newtheorem{thm}[subsection]{Theorem}

\newtheorem*{cor*}{Corollary}
\newtheorem{lemma}[subsection]{Lemma}
\newtheorem*{lem}{Lemma}
\newtheorem{propos}[subsection]{Proposition}

\newtheorem*{thm*}{Theorem}
\newtheorem*{thma*}{Theorem A}
\newtheorem*{thmb*}{Theorem B}
\newtheorem*{thmc*}{Theorem C}

\newtheorem{defn}[subsection]{Definition}
\newtheorem{cor}{Corollary}[subsection]

\newtheorem{conj}{Conjecture}

\newtheorem{rem}{Remark}[subsection]
\newtheorem*{claim}{Claim}
\newcounter{consta}
\renewcommand{\theconsta}{{L_{\arabic{consta}}}}
\newcommand{\consta}{\refstepcounter{consta}\theconsta}

\newcounter{constb}

\newcounter{constc}

\newcounter{constD}

\def\bbz{\mathbb{Z}}
\def\bbq{\mathbb{Q}}

\def\bbr{\mathbb{R}}

\def\bbc{\mathbb{C}}

\def\bbh{\mathbb{H}}
\def\bbn{\mathbb{N}}
\def\Gbf{\mathbb{G}}

\def\R{\bbr}

\def\Gbf{\mathbf{G}}
\def\Hbf{\mathbf{H}}

\def\H{\Hbf}
\def\G{\Gbf}

\def\vol{{\rm{vol}}}
\def\SL{{\rm{SL}}}
\def\PSL{{\rm{PSL}}}

\def\GL{{\rm{GL}}}
\def\PGL{{\rm{PGL}}}
\def\SO{{\rm{SO}}}

\def\qpl{{\mathcal S}}
\def\places{{\Sigma}}

\def\vare{\varepsilon}

\def\zg0{Z_{G_\omega}(s)}

\def\zg{Z_G(s)}

\def\be{\begin{equation}}
\def\ee{\end{equation}}

\def\dist{{\rm dist}}

\def\gfield{\mathsf L}
\def\gfd{\gfield}
\def\lf{\mathfrak l}

\def\sH{{}^\sigma\H}
\def\sG{{}^\sigma\G}
\def\sbG{{}^\sigma\G}

\def\P{{\bf Q}}
\def\sP{\P_\sigma}
\def\ogs{{\Omega}_{\Gamma,\rho}}
\def\fd{\mathfrak F}

\def\Gm{{\mathsf m}}
\def\dhyp{{\rm d}_h}
\def\diff{\operatorname{d}}
\def\dt{{\rm d}_{\mathcal T}}
\def\tc{\mathfrak C}
\def\dtv{{\rm d}_{\mathcal T_v}}
\def\sfc{{\mathsf S}}

\begin{document}

\title[Arithmeticity of $3$-manifolds]{Arithmeticity of hyperbolic $3$-manifolds containing infinitely many totally geodesic surfaces}

\author{G.~Margulis}
\address{G.M.: Yale University, Mathematics Dept., PO Box 208283, New Haven, CT 06520, USA}
\email{gregorii.margulis@yale.edu}

\author{A.~Mohammadi}
\address{A.M.:  Department of Mathematics, The University of California, San Diego, CA 92093, USA}
\email{ammohammadi@ucsd.edu}
\thanks{A.M.~acknowledges support by the NSF}

\maketitle


\begin{abstract}
We prove that if a closed hyperbolic $3$-manifold $M$ contains infinitely many totally geodesic surfaces,
then $M$ is arithmetic.   
\end{abstract}

\section{Introduction}\label{sec:intro}
Let $\G$ be a connected semisimple $\bbr$-group so that $\G(\R)$ has no compact factors.  
An irreducible lattice $\Gamma_0$ in $\G(\R)$ is called {\em arithmetic} if there exists a connected, non-commutative, almost $\bbq$-simple, $\bbq$-group ${\bf F}$ and an $\bbr$-epimorphism $\varrho:{\bf F}\to\G$ such that the Lie group $(\ker\varrho)(\bbr)$ is compact 
and $\Gamma_0$ is commensurable with $\varrho({\bf F}(\bbz))$, see~\cite[Ch.~IX]{Margulis-Book}.  

\medskip

Margulis~\cite{Margulis-Arith-Inv} proved the following. 

\begin{thma*}[Arithmeticity]
Let $\G$ be a connected semisimple $\bbr$-group so that $\G(\R)$ has no compact factors. 
Let $\Gamma_0$ be an irreducible lattice in $\G(\R)$. 
Assume further that ${\rm rank}_{\bbr}\G\geq 2$.
Then $\Gamma_0$ is arithmetic. 
\end{thma*}

Let $\Gamma_0$ and $\G(\bbr)$ be as in Theorem~A. 
One may reduce the proof of Theorem~A to the case where $\G$ is a group of adjoint type
defined over a finitely generated field $\gfield$ and $\Gamma_0\subset\G(\gfield)$ --- indeed using {\em local rigidity}, 
one may further assume that $\gfield$ is a number field.
The proof of Theorem~A is based on applying the following supperrigidity theorem, which was also proved in~\cite{Margulis-Arith-Inv}, to representations obtained from different embeddings of $\gfield$ into local fields. 

\begin{thmb*}[Superrigidity]
Let $\G$ be a connected semisimple $\bbr$-group. 
Let $\Gamma_0$ be an irreducible lattice in $\G(\bbr)$.
Assume further that ${\rm rank}_{\bbr}\G\geq 2$.
Let $\lf$ be a local field and let $\H$ be a connected, adjoint, absolutely simple $\lf$-group. 
Let $\rho:\Gamma_0\to\H(\lf)$ be a homomorphism so that
\begin{quote}
$\rho(\Gamma_0)$ is Zariski dense and is not bounded in $\H(\lf)$.
\end{quote}
Then $\rho$ extends uniquely to a continuous homomorphism $\tilde\rho:\G(\bbr)\to\H(\lf)$.
\end{thmb*}

It follows from the {\em weak approximation theorem} that if $\Gamma_0$ is an arithmetic group, 
the index of $\Gamma_0$ in ${\rm Comm}_{\G(\bbr)}(\Gamma_0)$ 
is infinite. Margulis proved the converse also holds, see~\cite[Ch.~IX]{Margulis-Book}.

\begin{thmc*}
Let $\G$ be a connected semisimple $\bbr$-group so that $\G(\R)$ has no compact factors. 
Let $\Gamma_0$ be an irreducible lattice in $\G(\R)$. 
Then $\Gamma_0$ is arithmetic if and only if the index of $\Gamma_0$ in ${\rm Comm}_{\G(\bbr)}(\Gamma_0)$ is infinite.  
\end{thmc*}

Supperrigidity and arithmeticity theorems continue to hold for certain rank one Lie groups, namely ${\rm Sp}(n,1)$ and $F_4^{-20}$,~\cite{Grom-Sch, Corl}. However, there are examples of non-arithmetic lattices in $\SO(n,1)$ for all $n$ and 
also in ${\rm SU}(n,1)$ for $n=1,2,3$.

\subsection*{Totally geodesic surfaces and arithmeticity}\label{sec:statement}
The connected component of the identity in the Lie group 
$\SO(3,1)$ is isomorphic to  
\[
{\rm Isom}^+(\bbh^3)\simeq\PGL_2(\bbc).
\]

Let $M=\bbh^3/\Gamma$ be a closed oriented hyperbolic $3$-manifold   
presented as a quotient of the hyperbolic space by the action of a lattice
\[
\Gamma\subset \PGL_2(\bbc).
\]

A {\em totally geodesic surface} in $M$ is a proper geodesic immersion of a closed hyperbolic surface into $M$.
It is well-known and easy to see that there can be at most countably many totally geodesic surfaces in $M$. 

Reid~\cite{Reid-Commens} showed that if $\Gamma$ is an arithmetic group, then either $M$ contains no totally geodesic surfaces or it contains infinitely many such surfaces --- see Theorem~C above. 
There are also known examples for both of these possibilities,~\cite{MaclachlanReid}. 
More recently, it was shown in~\cite{FLMS} that a large class of non-arithmetic manifolds contain only finitely many totally geodesic surfaces.

The following theorem is the main result of this paper. 

\begin{thm}\label{thm:main-Mnfld}
Let $M=\bbh^3/\Gamma$ be a closed hyperbolic $3$-manifold. 
If $M$ contains infinitely many totally geodesic surfaces,
then $M$ is arithmetic. That is: $\Gamma$ is an arithmetic lattice.  
\end{thm}

The statement in Theorem~\ref{thm:main-Mnfld} answers affirmatively a question asked by A.~Reid and C.~McMullen, see~\cite[\S8.2]{McR-Re} and~\cite[Qn.~7.6]{DHM-Prob}. 

\medskip

This paper is essentially a more detailed version of~\cite{MM}. 
More explicitly, several measurability statements were taken for granted in~\cite{MM}, we provide their more or less standard proofs here; moreover, this paper contains a more elaborate version of the proof of Proposition~\ref{thm:psi-circle-circle} 
when compared to the proof given in~\cite{MM}.  Overall, our goal has been to make this paper as self contained as possible. 

Shortly after the appearance of~\cite{MM}, Bader, Fisher, Miller, and Stover~\cite{BFMS} proved that if a finite volume hyperbolic $n$-manifold, $\bbh^n/\Gamma$, contains infinitely many maximal totally geodesic subspaces of dimension at least 2, then $\Gamma$ is arithmetic --- Theorem~\ref{thm:main-Mnfld} is a special case. Their proof and ours both use a superrigidity theorem to prove arithmeticity, but the superrigidity theorems and their proofs are quite different. 

\medskip

In view of Theorem~C, we get the following from Theorem~\ref{thm:main-Mnfld}. If $M=\bbh^3/\Gamma$ 
is a closed hyperbolic $3$-manifold which contains infinitely many totally geodesic surfaces, the index of $\Gamma$
in its commensurator is infinite. 

As was mentioned above the arithmeticity theorem for irreducible lattices in higher rank Lie groups was proved using the superrigidity Theorem~B. Similarly, Theorem~\ref{thm:main-Mnfld} follows from Theorem~\ref{thm:sup-rigid} below which is a rigidity type result.

\subsection*{A rigidity theorem} 

Let $\gfield$ be the number field and $\G$ the connected, semisimple $\gfield$-group of adjoint type associated to $\Gamma$,  see~\S\ref{sec:notation} also~\cite{Margulis-Book, MaclachlanReid}. In particular, $\gfield\subset\bbr$, 
$\Gamma\subset\G(\gfield)$, and $\G$ is $\bbr$-isomorphic to ${\rm PO}(3,1)$ --- 
the connected component of the identity in the Lie group $\G(\bbr)$ is isomorphic to $\PGL_2(\bbc)$.

Let $\qpl$ denote the set of places of $\gfield$. 
For every $v\in\qpl$, let $\gfield_v$ be the completion of $\gfield$ at $v$ and let $\places_v$ be the set of Galois embeddings $\sigma:\gfield\to\gfield_v$.

For any $v\in\qpl$ and any $\sigma\in\places_v$, we let $\sbG$ denote the algebraic group 
defined by applying $\sigma$ to the equations of $\G$. Let $v\in\qpl$ and $\sigma\in\places_v$, then $\sigma(\Gamma)\subset\sbG$ is Zariski dense.



\begin{thm}\label{thm:sup-rigid}
Let $M=\bbh^3/\Gamma$ 
be a closed hyperbolic $3$-manifold. Assume further that $M$ contains infinitely many totally geodesic surfaces.
Let $\gfield$ and $\G$ be as above. 

If $v\in\qpl$ and $\sigma\in\places_v$ are so that 
$\sigma(\Gamma)\subset\sG(\gfd_v)$ is unbounded,
then $\sigma$ extends to a continuous homomorphism from $\G(\bbr)$ to $\sG(\gfd_v).$
\end{thm}

Theorem~\ref{thm:main-Mnfld} follows from Theorem~\ref{thm:sup-rigid}. 
We will recall the argument from~\cite{Margulis-Arith-Inv} in~\S\ref{sec:proofs} --- indeed the group $\bf F$ in the definition of an arithmetic group is the Zariski closure of $\Gamma$ in the restriction of scalars group ${\rm R}_{\gfield/\bbq}(\G)$, see also~\cite[Ch.~IX]{Margulis-Book}.

\medskip

The proof of Theorem~\ref{thm:sup-rigid} is based on the study of certain $\Gamma$-equivariant measurable maps from 
$\partial\bbh^3=\mathbb S^2$ into projective spaces --- equivariant maps of this kind also play a pivotal role in the proof of the {\em strong rigidity} theorem by Mostow and the proof of the {\em superrigidity} theorem by Margulis. 

Indeed the proof in~\cite{Margulis-Arith-Inv} is based on showing that an a priori only measurable {\em boundary map} agrees with a rational map almost surely; this rationality is then used to find the desired continuous extension. Our strategy here is to show that if $M$ contains infinitely many totally geodesic surfaces, a certain $\Gamma$-equivariant measurable map on $\mathbb S^2$ is almost surely rational, see Proposition~\ref{thm:psi-circle-circle}. 
In ~\S\ref{sec:proofs} we use Proposition~\ref{thm:psi-circle-circle} to complete the proof of 
Theorem~\ref{thm:sup-rigid}, see~\cite{Margulis-Arith-Inv}. 

\medskip

We end the introduction by mentioning that in this paper the discussion is restricted to the case of closed hyperbolic $3$-manifolds; however, our method extends to the case of finite volume hyperbolic $3$-manifolds. 
Indeed our argument rests upon investigating certain properties of a cocycle which will be introduced in \S\ref{sec:basic-lemma}. 
The extension to finite volume hyperbolic $3$-manifolds requires some estimates for the growth rate of this cocycle.  
The desired estimates may be obtained using a similar, and simpler, version of systems of inequalities in~\cite{EMM-Upp}.

\subsection*{Acknowledgement} We would like to thank D.~Fisher, C.~McMullen, H.~Oh, and A.~Reid for their helpful comments on earlier versions of this paper.

\section{Preliminaries and notation}\label{sec:notation}
Let $G=\PGL_2(\bbc)$ and let $\Gamma\subset G$ be a lattice.
Let $X=G/\Gamma$ and let $\vol_X$ (or simply $\vol$) denote the $G$-invariant probability measure on $X$. 
We let $\pi$ denote the natural projection from $G$ to $X$. Also let $K={\rm SU}(2)/\{\pm I\}$.
 
We let $H=\PGL_2(\bbr)$. For every $t\in\bbr$, let
\be\label{eq:def-a-t}
a_t=\begin{pmatrix}e^{t/2}&0\\0&e^{-t/2}\end{pmatrix};
\ee
note that $a_t\in H$ for all $t\in\bbr$. For every $\theta\in[0,2\pi]$, $r_\theta\in\PGL_2(\bbr)$ denotes the rotation with angle $\theta$.

The bundle of oriented frames over $\bbh^3=K\backslash G$ may be identified with $G$.
The left action of $\{a_t:t\in\bbr\}$ on $G$ and $G/\Gamma$ induces 
the frame flow on the frame bundles of $\bbh^3$ and $M$, respectively. 
For any $g\in G$ the image of $Hg$ in $\bbh^3$ is a geodesic embedding of $\bbh^2$ into $\bbh^3$. 
In this setup, a totally geodesic surface in $M=K\backslash G/\Gamma$ lifts to a closed orbit of $H$ in $X$.

%
%

\subsection{The number field $\gfield$ and the $\gfield$-group $\G$}\label{sec:nf-gp}
Since $\Gamma$ is finitely generated, there exists a finitely generated field $\gfield$ and a connected, semisimple $\gfield$-group $\G$ 
of adjoint type so that $\gfield\subset\bbr$, $\Gamma\subset\G(\gfield)$, and $\G$ is $\bbr$-isomorphic to ${\rm PO}(3,1)$.
In view of local rigidity of $\Gamma$,~\cite[Thm.~0.11]{GarRag-FD}, 
$\gfield$ is indeed a number field,  see also~\cite{Selb, Weil-1, Weil-2}. 
We will refer to the pair $(\gfield,\G)$ as the number field and the group associated to $\Gamma$, see~\cite{Margulis-Book, MaclachlanReid}.

Let $\qpl$ denote the set of places of $\gfield$. 
For every $v\in\qpl$, let $\gfield_v$ be the completion of $\gfield$ at $v$ and let $\places_v$ be the set of Galois embeddings $\sigma:\gfield\to\gfield_v$.

With this notation we let $(v_0,{\rm id})$ be the pair which gives rise to the lattice $\Gamma$ in $G$ ---
recall that the connected component of the identity in the Lie group $\G(\bbr)$ is isomorphic to $\PGL_2(\bbc)$.  

For any $v\in\qpl$ and any $\sigma\in\places_v$, we let $\sbG$ denote the $\sigma(\gfield)$-group 
defined by applying $\sigma$ to the coefficients of the defining equations of $\G$. Let $v\in\qpl$ and $\sigma\in\places_v$, then $\sigma(\Gamma)\subset\sbG$ is Zariski dense.

Note that $\G$ is isomorphic to $\PGL_2\times\PGL_2$ over $\bbc$. More generally, for every $v\in\qpl$, 
there exists an extension $\lf_v/\gfield_v$ of degree at most 2 so that $\sbG$ splits over $\lf_v$. That is:
\be\label{eq:G-lfv}
\text{$\sbG$ is isomorphic to $\PGL_2\times\PGL_2$ over $\lf_v$.}
\ee

Let $\Delta\subset\Gamma$ be a non-elementary Fuchsian group; in the above notation, we have the following. 
Let $\H_\Delta$ be the Zariski closure of $\Delta$ in $\G$. Then the Lie group $\H_\Delta(\bbr)$ is locally isomorphic to $\SL_2(\bbr)$.

For every $v\in\qpl$ and every $\sigma\in\places_v$, let $\sH_\Delta\subset\sbG$ denote the $\sigma(\gfield)$-group obtained by applying $\sigma$ to the defining coefficients of $\H_\Delta$ --- note that $\sH_\Delta$ is the Zariski closure of $\sigma(\Delta)$ in $\sbG$.

\begin{lemma}\label{lem:proj}
Let $\Delta\subset\Gamma$ be a non-elementary Fuchsian group.
Assume that $\sigma(\Delta)$ is unbounded in 
$
\sbG(\lf_v)=\PGL_2(\lf_v)\times\PGL_2(\lf_v).
$
Then there exists some $g\in\PGL_2(\lf_v)$ so that
\[
\sH(\lf_v)\cap\{(h,ghg^{-1}):h\in\PGL_2(\lf_v)\}
\]
is a subgroup of index at most $8$ in $\sH(\lf_v)$.
\end{lemma}

\begin{proof}
The group $\sH$ is isogeneous to $\SL_2$, in particular, it is a proper algebraic subgroup of $\sbG$.
Moreover, $\sH$ intersects each factor of $\sbG$ trivially.
Indeed, $\sbG$ has trivial center; therefore, if $\sH$ intersects a factor non-trivially, 
then this intersection is normal in $\sH$ which implies that $\sH\vartriangleleft\sbG$. 
Now since $\H(\gfield)$ and $\G(\gfield)$
are Zariski dense in $\H$ and $\G$, respectively, we get that $\H\vartriangleleft\G$; this is a contradiction.

In consequence, there exists an $\lf_v$-group $\bf L$ which is an $\lf_v$-form of $\SL_2$, i.e., either arising from a division algebra 
or $\SL_2$, and an $\lf_v$-isogeny $\varphi:{\bf L}\to\sH$, so that $\sH(\lf_v)$ contains $\varphi({\bf L}(\lf_v))$
as a subgroup of finite index.  We claim $\bf L$ is indeed $\SL_2$. 
The group $\sigma(\Delta)$ is unbounded, therefore, the group $\sH(\lf_v)$ is unbounded. Hence, 
$\sH$ is $\lf_v$-isotropic, this implies the claim.

Since $\sH$ is a group of type $A_1$, the isogeny $\varphi$ is inner, i.e.,
it is induced by conjugation with an element $g\in \PGL_2(\lf_v)$. 

The claim regarding the index follows as $[\PGL_2(\lf):\SL_2(\lf)]\leq 8$ for any local field $\lf$ of characteristic zero.
\end{proof}

\subsection{Lemmas from hyperbolic geometry}\label{sec:hyp-geom}
In this section we will recrod some basic facts from hyperbolic geometry to be used in the sequel. 
Let us begin by recalling the hyperbolic law of cosines. 
Let $pqr$ be a hyperbolic triangle and let $\theta$ be the angle opposite to the edge $\overline{qr}$. Then
\be\label{eq:hyp-law-cos}
\cosh(\overline{qr})= \cosh(\overline{pq})\cosh(\overline{pr})-\cos(\theta)\sinh(\overline{pq})\sinh(\overline{pr})
\ee

Let $(\mathcal T, \dt)$ denote either a regular tree equipped with the usual path metric or a hyperbolic space equipped with the hyperbolic metric. We fix a base point $o\in\mathcal T$.

\begin{lemma}\label{lem:hyp-1}
Let $\{p_{n}: n=0,1,\ldots\}\subset\mathcal T$ with $p_0=o$. Assume that there exist some $L_1, L_2, N_0>1$ so that 
$\dt(p_{n},p_{(n+1)})\leq L_1$ for all $n$ and 
\[
\text{$\dt(p_n,o)\geq n/L_2\;\;$ for all $n>N_0$.}
\]
Then there exists a unique geodesic $\xi=\{\xi_t:t\in\bbr\}\subset\mathcal T$ with $\xi_0=o$ so that 
\[
p_n\to\xi_\infty\in\partial\mathcal T.
\]
Moreover, for every $\vare>0$ there exists some $N=N(L_1,L_2,N_0,\vare)$ so that for all $n>N$ we have
$\dt(p_n,\xi)\leq \vare n$. 
\end{lemma}

We will also need the following lemma; the proof of this lemma uses Lemma~\ref{lem:hyp-1} and Egorov's theorem. 

\begin{lemma}\label{lem:hyp-2}
Let $(\Theta,\vartheta)$ be a Borel probability space. 
Let $\psi:\Theta\to\partial\mathcal T$ and $u:\bbz^{\geq0}\times\Theta\to \mathcal T$ be two Borel maps satisfying the following.
 
\begin{enumerate}
\item $\vartheta(\psi^{-1}\{p\})=0$ for every $p\in\partial \mathcal T$. 
\item $u(0,\theta)=o$ for a.e.\ $\theta\in\Theta$.
\item There exists some $L_1\geq 1$ so that for a.e.\ $\theta\in\Theta$ we have 
\[
\text{$\dt(u(n,\theta),u(n+1,\theta)\leq L_1\;\;$ for all $n$.}
\]
\item There exists some $L_2\geq1$ and for a.e.\ $\theta\in\Theta$ there exists some $N_\theta$ so that 
\[
\text{$\dt(u(n,\theta),u(n+1,\theta)\geq n/L_2\;\;$ for all $n>N_\theta$.}
\]
\end{enumerate}
In particular, by Lemma~\ref{lem:hyp-1} we have $\{u(n,\theta)\}$ converges to a point in $\partial\mathcal T$ for a.e.\ $\theta\in\Theta$.
Assume further that 
\be\label{eq:limit-psi-infty}
\text{$u(n,\theta)\to\psi(\theta)\;\;$ for a.e.\ $\theta\in\Theta$}.
\ee
Let $\{\xi_t\}\subset\mathcal T$ be any geodesic with $\xi_0=o$.
There exists some $c'=c'(\psi,u,L_1,L_2)$ so that for all $t\in\R$ and all $n\in\bbn$ we have 
\[
\int_\Theta \dt(u(n,\theta),\xi_t)\diff\!\vartheta(\theta)>t+\frac{n}{4L_2}-c'.
\]
\end{lemma}

\begin{proof}
We explicate the proof when $\mathcal T$ is a hyperbolic space, the proof in the case of a tree is a simple modification.

For any $\theta\in\Theta$, let $q_\theta=u(n,\theta)$. By the hyperbolic law of cosines,~\eqref{eq:hyp-law-cos}, we have
\be\label{eq:hyp-law-psi-Theta}
\cosh(\overline{q_\theta\xi_t})= \cosh(\overline{o\xi_t})\cosh(\overline{oq_\theta})-\cos(\alpha)\sinh(\overline{o\xi_t})\sinh(\overline{oq_\theta})
\ee
where $\alpha$ is the angle between $\overline{o\xi_t}$ and $\overline{oq_\theta}$.

In view of our assumption~(1) and the compactness of $\partial\mathcal T$ for every $\vare>0$
there exists some $\delta>0$ so that for every $p\in\partial\mathcal T$ we have
\be\label{eq:using-assum-1-Theta}
\vartheta\Bigl(\psi^{-1}(\mathsf N_\delta(p))\Bigr)<\vare/2.
\ee

Fix some $0<\vare<1/8L_1L_2$. 
In view of Egorov's theorem, Lemma~\ref{lem:hyp-1} and~\eqref{eq:limit-psi-infty}, 
there exists $\Theta'\subset\Theta$ with $\vartheta(\Theta')>1-\vare/2$ 
and some $N_0\geq 1$ so that for all $\theta\in\Theta'$ and all $n>N_0$ we have 
\be\label{eq:unif-conv-Theta-psi}
\dt(u(n, \theta),o)\geq n/2L_2.
\ee

Let $\mathsf{E}:=\{\theta\in\Theta': \psi(\theta)\not\in\mathsf N_\delta(\xi_\infty)\}$
where $\xi_\infty=\lim_{t\to\infty}\xi_t\in\partial \mathcal T$. Then thanks to~\eqref{eq:using-assum-1-Theta} and
$\vartheta(\Theta')>1-\vare/2$ we have $\vartheta(\mathsf E)>1-\vare$.

Now using~\eqref{eq:hyp-law-psi-Theta} and~\eqref{eq:unif-conv-Theta-psi}, we conclude that 
\be\label{eq:law-cos-Theta-L2}
\dt(q_\theta,\xi_t)>t+\tfrac{n}{2L_2}-O_\delta(1)
\ee
for all $\theta\in\mathsf E$ and all $n>N_0$.

Note also that by our assumption~(3), we have 
$\dt(u(n,\theta),o)\leq L_1 n$ for all $\theta\in \Theta$. In particular, we get that
\be\label{eq:triv-boud-vs-cosine-Theta}
\dt(q_\theta,\xi_t)\geq t-L_1n
\ee
for all $\theta\in \Theta$ and all $n\in\bbn$.

Now, splitting the integral over $\mathsf E$ and its compliment 
and using the estimates in~\eqref{eq:law-cos-Theta-L2} and~\eqref{eq:triv-boud-vs-cosine-Theta}, we obtain the following:
\begin{align*}
\int_{\Theta}\dt(q_\theta,\xi_t)\diff\!\vartheta(\theta)&>(1-\vare)(t+\tfrac{n}{2L_2}-O_\delta(1))+ (t-L_1n)\vare\\
{}^{\vare<1/4\leadsto}&> t+\frac{n}{4L_2}+(\frac{n}{8L_2}-L_1\vare n)-O_\delta(1)\\
{}^{\vare<1/8L_1L_2\leadsto}&> t+\frac{n}{4L_2}-O_{\delta}(1).
\end{align*}
This completes the proof if we let $c'=O_\delta(1)$.
\end{proof}

It is worth mentioning that the above lemmas hold for any proper, complete, ${\rm CAT}(-1)$ space. 

\subsection{Action of $\Gamma$ on varieties}\label{sec:act-variety}
Let $\mathcal C$ be the space of circles in $\mathbb S^2=\partial\bbh^3$; 
the space $\mathcal C$ is equipped with a natural $\PGL_2(\bbc)$-invariant measure $\sigma$.
Let $\Gm$ denote the Lebesgue measure on $\mathbb S^2$.
For every $C\in\mathcal C$, let $\Gm_C$ be the Lebesgue measure on $C$. 

Here and in what follows, by a {\em non-atomic} measure we mean a measure without any atoms.  

The following general lemma will be used in the sequel.

\begin{lemma}~\label{lem:PsiGmZd}\label{cor:Fubini-Circles}
Let $\lf$ be a local field. Let $\H$ be a connected $\lf$-group which acts $\lf$-rationally on an irreducible $\lf$-variety $\bf V$. 
Assume further that $\H$ acts transitively on ${\bf V}$.
Let $\rho:\Gamma\to\H(\lf)$ be a homomorphism so that $\rho(\Gamma)$ is Zariski dense in $\H$ and let 
$
\Psi: \mathbb S^2\to{\bf V}(\lf)
$ 
be a $\Gamma$-equivariant measurable map. Then
\begin{enumerate}
\item $\Gm\Bigl(\{\xi\in\mathbb S^2:\Psi(\xi)\in{\bf W}(\lf)\}\Bigr)=0$ for any proper subvariety $\bf W\subset{\bf V}$.
\item For $\sigma$-a.e.\ $C\in\mathcal C$ and every $p\in\mathbf V(\lf)$ we have 
\[
\Gm_C\Bigl(\{\xi\in C:\Psi(\xi)=p\}\Bigr)=0.
\]
\end{enumerate}
\end{lemma}

\begin{proof}
Part~(1) is well-known, see~\cite[Lemma 4.2]{Golds-Marg} also~\cite[Ch.~VI, Lemma 3.10]{Margulis-Book} and~\cite{Fu-Boundary}.

We prove part~(2). 
Let $\mathcal A=\{(C,\xi,\xi')\in \mathcal C\times \mathbb S^2\times \mathbb S^2:\xi,\xi'\in C\}$.
Equip $\mathcal A$ with the natural measure arising from the $G$-invariant measure $\sigma$ on $\mathcal C$ and the measure $\Gm_C$ on $C\in\mathcal C$.  

Define 
$\Phi: \mathcal A\to \mathbf V(\lf) \times \mathbf V(\lf)$ 
by $\Phi(C,\xi,\xi')=(\Psi(\xi),\Psi(\xi'))$. Assume the claim in part~(2) fails. Then 
\[
\mathcal B:=\Phi^{-1}(\{(p,p): p\in \mathbf V(\lf)\})\subset \mathcal A
\] 
has positive measure.

In view of Fubini's theorem then 
the projection of $\mathcal B$ onto $\mathbb S^2\times \mathbb S^2$ contains a positive measure subset.
That is: there exists a positive measure subset of $\mathbb S^2\times \mathbb S^2$ which gets mapped into 
$\{(p,p): p\in \mathbf V(\lf)\}$.
This implies that $\Psi$ maps a positive measure subset of $\mathbb S^2$ to a point, which contradicts part~(1).
\end{proof}

For any compact subset $E\subset \bbh^3/\Gamma$, let  
$\mathcal C_E=\{C\in\mathcal C:$ the convex hull of $C$ intersects $E$ nontrivially$\}$.


\begin{lemma}\label{lem:Egorov-Noatom}
Let the notation be as in Lemma~\ref{lem:PsiGmZd}. 
Further, assume that ${\bf V}(\lf)$ is compact and that it is equipped 
with a metric. For every $p\in{\bf V}(\lf)$ and every $r>0$, let $\mathsf N_r(p)$ 
denote the open ball of radius $r$ in this metric.

Let $E\subset\bbh^3/\Gamma$ be a compact subset with positive measure.
For every $\vare>0$ there exists a compact subset $\mathcal C_{E,\vare}\subset \mathcal C_E$ 
with $\sigma(\mathcal C_E\setminus\mathcal C_{E,\vare})\ll_E\vare^{32}\sigma(\mathcal C_E)$ 
and some $\delta>0$ with the following property. 
For every $C\in \mathcal C_{E,\vare}$ and every $p\in{\bf V}(\lf)$ we have 
\[
\Gm_C\Bigl(\{\xi\in C: \Psi(\xi)\in\mathsf N_{\delta}(p)\}\Bigr)<\vare\Gm_C(C).
\] 
\end{lemma}

\begin{proof}

First note that $\mathcal C_E$ is a compact subset of $\mathcal C$ with $\sigma(\mathcal C_E)>0$. 
Note also that, up to a null set, 
we may identify the space of circles $\mathcal C$ with $\partial\bbh^3\times (0,1]$.

Therefore, ignoring a possible null subset, there exist $r, r'>0$ so that 
\[
\mathcal C_E\subset\partial\bbh^3\times [r,1]=:\mathcal C_r
\] 
and
$\Gm(\mathcal C_E)\geq r'\Gm(\mathcal C_r)$.

By Lusin's theorem, there exists a compact subset $D\subset\partial\bbh^3$ with $\Gm(D)>1-\vare^{64}$ so that 
$\Psi|_D$ is continuous. 

For every $\xi\in\partial\bbh^3$, set 
\[
I_\xi:=\{t\in[r,1]:\Gm_{C_t(\xi)}(C_t(\xi)\cap D)<(1-\vare^{32})\Gm_{C_t(\xi)}(C_t(\xi))\}
\]
where $C_t(\xi)$ is the circle with radius $t$ centered at $\xi$.
By Fubini's theorem, we have: $|I_\xi|\ll_r\vare^{32}$ for every $\xi\in\partial\bbh^3$. 
Therefore,
\[
\sigma\Bigl(\{C\in\mathcal C_r: \Gm_C(C\cap D)<(1-\vare^{32})\Gm_C(C)\}\Bigr)\ll_r \vare^{32}.
\]

Let $\hat{\mathcal C}_{E,\vare}=\{C\in\mathcal C_E: \Gm_C(C\cap D)\geq (1-\vare^{32})\Gm_C(C)\}$. In view of the above estimate and since $\mathcal C_E\subset\mathcal C_r$, we have $\sigma(\mathcal C_E\setminus\hat{\mathcal C}_{E,\vare})\ll_{r,r'}\vare^{32}$.

Let $\mathcal C'\subset\mathcal C$ be the conull subset where Lemma~\ref{cor:Fubini-Circles}(2) holds true.
Let $\mathcal C_{E,\vare}$ be a compact subset of $\hat{\mathcal C}_{E,\vare}\cap\mathcal C'$ so that 
$\sigma(\mathcal C_E\setminus\mathcal C_{E,\vare})\ll_{r,r'}\vare^{32}\sigma(\mathcal C_E)$. 
In particular, $\mathcal C_E$ satisfies the first claim in the lemma.
We now verify that it also satisfies the second claim.

Assume contrary to the claim that for every $n$ there is some 
$p_n\in{\bf V}(\lf)$ and some $C_n\in\mathcal C_{E,\vare}$ so that
\be\label{eq:cont-psi}
\Gm_{C_n}\Bigl(\{\xi\in C_n: \Psi(\xi)\in\mathsf N_{1/n}(p_n)\}\Bigr)>\vare\Gm_{C_n}(C_n). 
\ee

Passing to a subsequence, if necessary, we assume that $C_n\to C\in\mathcal C_{E,\vare}$ and $p_n\to p\in{\bf V}(\lf)$ 
--- recall that ${\bf V}(\lf)$ is compact.

For each $n$, let $C_n':=\{\xi\in C_n: \Psi(\xi)\in\mathsf N_{1/n}(p_n)\}\cap D$.
In view of the fact that $\mathcal C_{E,\vare}\subset\hat{\mathcal C}_{E,\vare}$ and using~\eqref{eq:cont-psi}, 
we get that 
\[
\Gm_{C_n}(C_n')\geq \vare\Gm_{C_n}(C_n)/2.
\]

Let $C':=\lim\sup C_n'$; then $C'\subset C\cap D$ and $\Gm_C(C')\geq\vare\Gm_C(C)/2$.
Moreover, for every $\xi\in C'$ there exist some $\xi_n\to\xi$ with $\xi_n\in C_n'$.
Since $\Psi$ is continuous on $D$, we get that $\Psi(\xi)=p$, i.e., $\Psi(C')=p$.
This contradicts the fact that $C'\in\mathcal C_{E,\vare}\subset\mathcal C'$ and finishes the proof.
\end{proof}

\section{A $\Gamma$-equivariant circle preserving map}
In this section we state one of the main results of this paper, Proposition~\ref{prop:psi-circle-circle}.
We then complete the proofs of Theorem~\ref{thm:sup-rigid} and Theorem~\ref{thm:main-Mnfld} 
using Proposition~\ref{prop:psi-circle-circle}.

Let the notation be as in \S\ref{sec:nf-gp}.
In particular, $\gfield$ is a number field and $\G$ is an $\gfield$-group.
For every $v\in\qpl$ and $\sigma\in\places_v$, let $\lf_v$ denote $\bbc$ if $v$ is an Archimedean place, 
and an extension of degree at most 2 of $\gfd_v$ so that $\sG$ is $\lf_v$-split if $v$ is a non-Archimedean place. 
Recall from~\eqref{eq:G-lfv} that 
\[
\text{$\sG$ is $\lf_v$-isomorphic to $\PGL_2\times\PGL_2$.}
\]

Let $B_v$ denote the group of upper triangular matrices in $\PGL_2(\lf_v)$.
For every $g\in \PGL_2(\lf_v)$ define $\mathfrak C_g$ to be the image of 
\[
\{(h,ghg^{-1}): h\in\PGL_2(\lf_v)\}
\] 
in $B_v\backslash\PGL_2(\lf_v)\times B_v\backslash\PGL_2(\lf_v)={\mathbb P}\lf_v\times{\mathbb P}\lf_v$. 

Note that for every $g$, $\mathfrak C_g$ is the graph of the linear fractional transformation $g:\mathbb P\lf_v\to\mathbb P\lf_v$.



Recall from \S\ref{sec:act-variety} that $\mathcal C$ denotes the space of circles in $\mathbb S^2=\partial\bbh^3$; 
the space $\mathcal C$ is equipped with a natural $\PGL_2(\bbc)$-invariant measure $\sigma$.
Recall also that $\Gm$ denotes the Lebesgue measure on $\mathbb S^2$, and that
for every $C\in\mathcal C$, $\Gm_C$ denotes the Lebesgue measure on $C$.

\begin{propos}\label{prop:psi-circle-circle}\label{thm:psi-circle-circle}
Assume $X=G/\Gamma$ contains infinitely many closed $H$-orbits.
Let $\lf_v$ and $\G$ be as above and assume that 
\[
\text{$\sigma(\Gamma)\subset\PGL_2(\lf_v)\times\PGL_2(\lf_v)$ is unbounded.}
\]
There exists a $\Gamma$-equivariant measurable map 
\[
\Psi:\mathbb S^2\to {\mathbb P}\lf_v\times{\mathbb P}\lf_v
\]
with the following properties. 
\begin{enumerate}
\item For a.e.\ $C\in\mathcal C$ we have $(\Psi|_{C})_*\Gm_C$ is non-atomic.
\item For a.e.\ $C\in\mathcal C$ there exists some $g_C\in\PGL_2(\lf_v)$ so that the essential image of 
$\Psi|_C$ is contained in $\tc_{g_C}$.
\end{enumerate}
\end{propos}

This proposition will be proved in \S\ref{sec:proof-of-prop}.
Our goal in this section is to complete the proofs of Theorem~\ref{thm:sup-rigid} and Theorem~\ref{thm:main-Mnfld} using Proposition~\ref{thm:psi-circle-circle}.

\medskip

In the sequel by an {\em inversion} of a circle $C$ we mean a linear fractional transformation on $C$ 
of order $2$ where $C$ is identified with $\mathbb {P}\bbr$; similarly we define an inversion of $\tc_g$.

\begin{lemma}\label{lem:equi-map-circle}
Let $\varphi: C\to\tc_g$ be a Borel measurable map so that the essential image of $\varphi$ has at least three points. 
Let $\mathcal I$ be a subset of inversions on $C$ which generates $\PSL_2(\bbr)$.
Assume further that there exists a Borel map $f$ from $\mathcal I$ into the set of inversions on 
$\tc_g$ which satisfies the following:
\be\label{eq:fJ-phi}
\text{for any $\iota\in\mathcal I$, $\;\varphi\circ \iota=f(\iota)\circ\varphi\;\;$ $\Gm_C$--a.e.\ on $C$.}
\ee
Then $f$ extends to a continuous homomorphism from $\PSL_2(\bbr)$ into $\PGL_2(\lf_v)$.
\end{lemma}

\begin{proof}
First note that since the essential image of $\varphi$ has at least three points,
any linear fractional transformation on $\tc_g$ is uniquely determined by its restriction to 
the essential image of $\varphi$. 

In view of this and~\eqref{eq:fJ-phi} the map
$\iota_1\circ\cdots\circ \iota_n\mapsto f(\iota_1)\circ\cdots\circ f(\iota_n)$ is a well-defined measurable homomorphism 
from the group generated by $\mathcal I$ into $\PGL_2(\lf_v)$.

The claim follows from this as $\mathcal I$ generates $\PSL_2(\bbr)$ and any 
measurable homomorphism is continuous, see e.g.~\cite[Ch.~VII, Lemma 1.4]{Margulis-Book}. 
\end{proof}

Let $(\xi, \xi')\in \mathbb S^2\times \mathbb S^2$, with $\xi\neq \xi'$.
We let $C_t(\xi,\xi')$, $t\in I_{\xi,\xi'}\subset\bbr$ denote the one parameter family of circles in $\mathbb S^2$ passing through $\xi$ and $\xi'$.

Given a triple $(\xi,\xi',C)\in \mathbb S^2\times \mathbb S^2\times\mathcal C$ we say $\{\xi,\xi'\}$ and $C$ are {\em linked} if $\xi$ and $\xi'$ belong to different connected components of $\mathbb S^2\setminus C$. For every circle $C\in\mathcal C$, let 
\be\label{eq:def-EC}
E_C\subset \mathbb S^2\times \mathbb S^2
\ee 
denote the set of $(\xi,\xi')\in \mathbb S^2\times \mathbb S^2$ so that $\{\xi,\xi'\}$ and $C$ are linked.

\begin{lemma}\label{lem:inversion}
Let $C\in\mathcal C$ and $(\xi,\xi')\in \mathbb S^2\times \mathbb S^2$; assume that $\{\xi,\xi'\}$ and $C$ are linked.
Then the one parameter family $\{C_t(\xi,\xi'):t\in I_{\xi,\xi'}\}$ defines an inversion on the circle $C$.
\end{lemma}

\begin{proof}
Define $\iota_{\xi,\xi'}$ on $C$ as follows. Let $p\in C$, there exists a unique circle $C'\in C_t(\xi,\xi')$
which passes through $\xi,\xi',p$. Then $C'\cap C=\{p,q\}$. Let $\iota_{\xi,\xi'}(p)=q$ --- note that $\iota_{\xi,\xi'}$ has order two.
The map $\iota_{\xi,\xi'}$ is an inversion on $C$. 
This could be seen as follows: we may assume $C$ is the unit circle in the plane, $\xi'=\infty$ and $\xi=(0,b)$ for some $0\leq b<1$. Let ${\rm pr}$ denote the stereographic projection of $C$ onto the line $\{y=b\}\cup\{\infty\}$. Then ${\rm pr}(0,1)=\infty$ and $\iota_{\xi,\xi'}(0,1)=(0,-1)$. Moreover, if ${\rm pr}\circ\iota_{\xi,\xi'}(p)=(a,b)$ with $a\neq0$, then ${\rm pr}\circ\iota_{\xi,\xi'}(q)=(\tfrac{b^2-1}{a},b)$.
\end{proof}

\begin{rem}\label{rk:inversion-target}\rm
We also need an analogue of Lemma~\ref{lem:inversion} in the target space, i.e., for $\tc_g$. 
This can be seen by a direct computation which involves solving a quadratic equation. As was done in the proof of Lemma~\ref{lem:inversion}, one may also simplify this computation as follows: 
We may reduce to the case where the graph is given by $zw=1$, $\xi=(\infty, \infty)$, $\xi'=(r,s)$ with $r,s$ 
both finite and every line $az+b$ through $\xi'$ intersects $zw=1$. Then, on the finite points of the graph, 
the inversion is given by $z\mapsto -z-\frac{b}{a}$.   
\end{rem}

\begin{lemma}\label{lem:inv-generate}
Let $E'\subset E_C$ be a set with positive measure.
Then the group generated by the set of inversions induced by $(\xi,\xi')\in E'$, see Lemma~\ref{lem:inversion}, contains
$\PSL_2(\bbr)$.
\end{lemma}

\begin{proof}
Let $\mathcal I'$ be the set of inversions induced by $E'$. Then 
\[
\mathcal I'\subset{\rm Inv}:=\left\{\begin{pmatrix}a & b\\ c& -a\end{pmatrix}: a,b,c\in\bbr, a^2+bc=-1\right\},
\] 
and if we equip ${\rm Inv}$ with the Lebesgue measure, $\mathcal I'$ has positive measure. 
Any set of positive measure in ${\rm Inv}$ generates $\PSL_2(\bbr)$. 

To see this, note that for a.e.\ $(g,g')\in{\rm Inv}$
we have $g{\rm Inv}\cap g'{\rm Inv}$ is one dimensional. Therefore, there are $g,g'\in\mathcal I'$
so that $Q:=(g\mathcal I')(g'\mathcal I')$ has positive measure in $\PSL_2(\bbr)$. Now $Q^{-1}Q$ contains an open neighborhood of the identity and $\PSL_2(\bbr)$ is connected, hence, $Q^{-1}Q$ generates a subgroup which contains $\PSL_2(\bbr)$.
\end{proof}

Let the notation be as in Proposition~\ref{thm:psi-circle-circle}. In particular, 
\[
\Psi: \mathbb S^2\to{\mathbb P}\lf_v\times{\mathbb P}\lf_v
\] 
is a $\Gamma$-equivariant measurable map in Proposition~\ref{thm:psi-circle-circle}.


\begin{lemma}\label{lem:map-Ccal-Meas}
Let $\mathcal V=\{(C,\theta,\theta')\in\mathcal C\times\mathbb S^2\times\mathbb S^2: \theta,\theta'\in C\}$.
Then $\mathcal V$ is a subvariety of $\mathcal C\times\mathbb S^2\times\mathbb S^2$;
equip $\mathcal V$ with the natural measure. 
There exists a conull Borel measurable subset
\[
\mathcal A\subset \mathcal V
\] 
so that $\Psi(\theta), \Psi(\theta')$ belong to the essential image of $\Psi|_C$ for all $(C,\theta,\theta')\in\mathcal A$,
where $\Psi$ is as in Proposition~\ref{thm:psi-circle-circle}.
\end{lemma}

\begin{proof}
Since $\Psi$ is measurable, for any $n\in\bbn$ there exists a compact subset $D_n\subset\mathbb S^2$ with
$\Gm(\mathbb S^2\setminus D_n)<1/n$ so that $\Psi|_{D_n}$ is continuous; we may also assume that $D_1\subset D_2\subset\cdots$. 

For every $m\in\bbn$, let $\mathcal C_m$ denote the set of circles with radius $\geq1/m$.
Then $\mathcal C_m$ is a compact subset of $\mathcal C$ and $\mathcal C=\cup\mathcal C_m$.
For every $m,n\in\bbn$ let $\mathcal C_{m,n}=\{C\in\mathcal C_m: \Gm_C(C\cap D_n)>0\}$.
Then $\cup_n \mathcal C_{m,n}$ is a conull Borel subset of $\mathcal C_m$ for every $m$.

Let $C\in\mathcal C_{m,n}$ and let $\theta\in C\cap D_n$ be a density point for $C\cap D_n$. Then by continuity of $\Psi|_{D_n}$, 
for every $r>0$, we have $\Psi^{-1}(\mathsf N_r(\Psi(\theta))\cap D_n\cap C$ is an open subset of $D_n\cap C$; 
since $\theta$ is a density point of $C\cap D_n$,
we get that $\Gm_C(\Psi^{-1}(\mathsf N_r(\Psi(\theta))\cap D_n\cap C)>0$. 
In particular, $\Psi(\theta)$ belongs to the essential image of $\Psi|_C$. 

Define $\mathcal A_{m,n}'=\{(C,\theta,\theta'): C\in\mathcal C_{m,n}, \theta,\theta'\in C\cap D_n\}$. Then $\mathcal A'_{m,n}$ is a Borel
subset of $\mathcal V$. Using a countable basis of open subsets, we see that 
\[
\mathcal A_{m,n}=\{(C,\theta,\theta')\in\mathcal A'_{m,n}: \theta,\theta'\text{ are density points of } C\cap D_n\}
\] 
is a conull subset of $\mathcal A_{m,n}'$. 

This in view of the above argument implies that $\mathcal A=\cup_{m,n}\mathcal A_{m,n}$ satisfies the claim in the lemma.
\end{proof}

%
%

\begin{lemma}\label{lem:Fubini-main-prop-1}
Let the notation be as in Proposition~\ref{thm:psi-circle-circle}. 
For a.e.\ $C\in\mathcal C$, there is a conull subset $E'_C\subset E_C$, see~\eqref{eq:def-EC}, 
so that for all $(\xi,\xi')\in E'_C$ the following hold.
\begin{enumerate}
\item The essential image of $\Psi|_C$ is a subset of $\tc_{g_C}$ and is an infinite subset. 
\item $\Psi(\xi)\neq \Psi(\xi')$. 
\item For a.e.\ $t\in I_{\xi,\xi'}$, there exists $\tc_t$ so that the essential image of $\Psi|_{C_t(\xi,\xi')}$ contained in $\tc_t$.
\item Let $C\cap C_t(\xi,\xi')=\{\theta,\theta'\}$. Then $\Psi(\theta)\neq\Psi(\theta')$, moreover, they both belong to the essential image of $\Psi|_C$. 
\end{enumerate}
\end{lemma}

\begin{proof}
Let $\mathcal C'\subset\mathcal C$ be a conull subset 
where Proposition~\ref{thm:psi-circle-circle}(1) and~(2) hold. 
Then, for any $C\in\mathcal C'$ part~(1) in this lemma holds.

Applying Lemma~\ref{lem:PsiGmZd}(1) with $\Psi$ and $\Psi(\xi)$ (for every $\xi\in \mathbb S^2$ so that $\Psi(\xi)$ is defined),
the set of $\xi'\in \mathbb S^2$ such that $\Psi(\xi')=\Psi(\xi)$ is a null set.
That is: for a.e.\ $(\xi,\xi')\in \mathbb S^2\times \mathbb S^2$ we have $\Psi(\xi)\neq\Psi(\xi')$. 

Note also that since $\mathcal C'$ is conull, for a.e.\ $(\xi,\xi')\in \mathbb S^2\times \mathbb S^2$ we have
\[
C_t(\xi,\xi')\in\mathcal C'\quad\text{ for a.e.\ $t\in I_{\xi,\xi'}$.}
\]

In consequence, for a.e.\ $(\xi,\xi')\in E_C$ parts~(2) and~(3) hold true.
We now show that~(4) also holds for a.e.\ $C\in\mathcal C'$ and a.e.\ $(\xi,\xi')\in E_C$.

Define $\mathcal A':=\mathcal A\cap {\rm pr}^{-1}(\mathcal C')$
where ${\rm pr}:\mathcal V\to\mathcal C$ is the projection map and $\mathcal A$ is as in Lemma~\ref{lem:map-Ccal-Meas}. 
Let $\mathcal B'\subset \mathcal A' \times \mathbb S^2\times \mathbb S^2\times \bbr$
be the set of points $(C,\theta,\theta', \xi,\xi',t)$ where $(C,\theta,\theta')\in\mathcal A'$, $(\xi,\xi')\in E_C$, and $t\in I_{\xi,\xi'}$ --- then $\mathcal B'$ is Borel subset. Let
\[
\mathcal B:=\Bigl\{(C,\theta,\theta', \xi,\xi',t)\in\mathcal B': C\cap C_t(\xi,\xi')=\{\theta,\theta'\}\Bigr\};
\]
note that $\mathcal B$ is also a Borel set. Indeed, $\mathcal B$ is the inverse image of the diagonal in 
$\Bigl((\mathbb S^1\times \mathbb S^1)/(\bbz/2)\Bigr)\times \Bigl((\mathbb S^1\times \mathbb S^1)/(\bbz/2)\Bigr)$ under the map 
$(C,\theta,\theta', \xi,\xi',t)\mapsto (\{\theta,\theta'\}, C\cap C_t(\xi,\xi'))$ where we identified $C$ with $\mathbb S^1$.

By the definition of $\mathcal B$ and Lemma~\ref{lem:map-Ccal-Meas}, 
for any $(C,\theta,\theta', \xi,\xi',t)\in\mathcal B$ we have $\Psi(\theta)$ and $\Psi(\theta')$
belong to the essential image of $\Psi|_C$.

Observe that if we fix $(C,\theta)$ and vary $(\xi,\xi')\in E_C$ and $t\in I_{\xi,\xi'}$, we cover every $\theta\neq\theta'\in C$.
It thus follows, from the implicit function theorem, that for every $C$, a.e.\ $(\xi,\xi')\in E_C$, and a.e.\ $t\in I_{\xi,\xi'}$, 
there exists some $(\theta,\theta')$ so that $(C,\theta,\theta',\xi,\xi', t)\in\mathcal B$ satisfies the second claim in part~(4) holds. 

We now show the first claim in~(4) also holds, possibly after removing another null subset.
Let $\check{\mathcal B}\subset\mathcal B$ be the set of $(C,\theta,\theta', \xi,\xi',t)\in\mathcal B$ so that $\Psi(\theta)=\Psi(\theta')$. 
We claim that $\check{\mathcal B}$ is a null set; this finishes the proof of~(4) and the lemma.

Assume to the contrary that $\check{\mathcal B}$ has positive measure.
By Fubini's theorem, thus, there exists some 
$(C,\theta)$ so that
\[
\check{\mathcal B}_{(C,\theta)}:=\{(\theta',\xi,\xi',t): (C,\theta,\theta',\xi,\xi', t)\in \check{\mathcal B}\}
\] 
has positive measure. 

Using the implicit function theorem, we get 
a subset $J\subset C$ with $\Gm_C(J)>0$ so that 
$\Psi(\theta')=\Psi(\theta)$ for all $\theta'\in J$. This contradicts the fact that $C\in\mathcal C'$ --- recall that Proposition~\ref{thm:psi-circle-circle}(1) holds for all $C\in \mathcal C'$. 
\end{proof}

\begin{lemma}\label{lem:Psi-rational}
Let the notation be as in Proposition~\ref{thm:psi-circle-circle}. In particular, 
\[
\Psi: \mathbb S^2\to{\mathbb P}\lf_v\times{\mathbb P}\lf_v
\] 
is a $\Gamma$-equivariant measurable map which satisfies parts~(1) and~(2) in Proposition~\ref{thm:psi-circle-circle}.
Then $\lf_v=\bbc$ and $\Psi$ agrees with a rational map from $\mathbb S^2$ into $\mathbb P\bbc\times\mathbb P\bbc$ almost everywhere.
\end{lemma}

\begin{proof}
Let $C\in\mathcal C$ be so that Lemma~\ref{lem:Fubini-main-prop-1} holds true and let $E'_C\subset E_C$ be as in loc.\ cit.
In particular the essential image of $\Psi|_C$ belongs to $\tc_{g_C}$. Let $\mathcal I$ be the collection of inversions of $C$ obtained 
by $(\xi,\xi')\in E'_C$, see Lemma~\ref{lem:inversion}. Then by Lemma~\ref{lem:inv-generate}, $\mathcal I$ generates $\PSL_2(\bbr)$.

Recall that the essential image of $\Psi|_C$ in $\tc$ is infinite, therefore, an inversion is uniquely determined by its restriction to the essential image of $\Psi$. By Lemma~\ref{lem:Fubini-main-prop-1} and Remark~\ref{rk:inversion-target},
$\Psi$ induces a map $f$ from $\mathcal I$ into the set of inversions on $\tc$. 

Since the essential image of $\Psi|_C$ in $\tc$ is infinite,
we get from Lemma~\ref{lem:equi-map-circle} that $f$ extends to a continuous homomorphism from $\PSL_2(\bbr)$ 
into $\PGL_2(\lf_v)$. 
Such a homomorphism can only arise from algebraic constructions as follows: There exists a continuous 
homomorphism of fields $\vartheta: \bbr\to\lf_v$ and an isomorphism of algebraic groups $\varphi: {}^\vartheta\PGL_2\to\PGL_2$ 
so that $f(g)=\varphi(\vartheta^0(g))$ for all $g\in\PSL_2(\bbr)$ where $\vartheta^0:\PGL_2(\bbr)\to\PGL_2(\vartheta(\bbr))$ is the isomorphism induced by $\vartheta$, see~\cite[Ch.~I, \S1.8]{Margulis-Book}. 

Since there are no monomorphism from $\bbr$ into non-Archimedean local fields, continuous or not, we get that $\lf_v=\bbc$.  

We now show that $\Psi$ agrees with a rational map\footnote{In the case at hand this assertion can be proved using more elementary arguments, e.g.\ by choosing three parallel circles.} almost surely. Let $C$ and $C'$ 
be two circles which intersect at two points and both satisfy Lemma~\ref{lem:Fubini-main-prop-1} --- the set of intersecting circles has positive measure in $\mathcal C$, hence, two such circles exist. Let $C\cap C'=\{\xi,\xi'\}$ where $\xi\neq \xi'$.
Using the stereographic projection of $\mathbb S^2$ with $\xi$ as the pole, 
we get coordinates on $\bbr^2$ induced by two lines $\ell$ and $\ell'$ corresponding $C$ and $C'$, respectively, 
which intersect at the image of $\xi'$. 
Thus $\Psi$ induces a measurable map in two variables which is rational in each variable.
In view of~\cite[Lemma 17]{Margulis-Arith-Inv}, we thus get that $\Psi$ agrees with a rational map almost surely.
\end{proof}

\subsection{Proofs of the main theorems}\label{sec:proofs}
In this section we complete the proofs of Theorem~\ref{thm:sup-rigid} and Theorem~\ref{thm:main-Mnfld}
assuming Proposition~\ref{thm:psi-circle-circle}.

\begin{proof}[Proof of Theorem~\ref{thm:sup-rigid}]
In view of Lemma~\ref{lem:Psi-rational} we may assume $\lf_v=\bbc$ and $\Psi$ agrees with a rational map almost surely. 
Theorem~\ref{thm:sup-rigid} follows from this by~\cite[\S1.3]{Margulis-Arith-Inv} as the action of $\PGL_2\times \PGL_2$ on its boundary is strictly effective.
\end{proof}

\begin{proof}[Proof of Theorem~\ref{thm:main-Mnfld}]
We recall the argument from~\cite[Proof of Thm.~1, p.\ 97]{Margulis-Arith-Inv}. 
Let the notation be as in~\S\ref{sec:nf-gp}; in particular, 
$v_0$ and $\sigma={\rm id}$ are the place and the embedding which give rise to the lattice $\Gamma$ in $\PGL_2(\bbc)$.
By Theorem~\ref{thm:sup-rigid} for any $(v,\sigma)\neq (v_0,{\rm id})$ we have $\sigma(\Gamma)$ is 
bounded in $\G(\gfield_v)$.

Let $\G'={\rm R}_{\gfield/\bbq}(\G)$ where ${\rm R}_{\gfield/\bbq}$ is the restriction of scalars.
Then $\G'(\bbr)$ is naturally identified with $\prod\G(\gfield_v)$ where the product is taken over all the Archimedean places.
Let $\varphi(\Gamma)$ denote the image of $\Gamma$ in $\G'$ --- note that 
$\varphi(\Gamma)$ is isomorphic to $\Gamma$ and $\varphi(\Gamma)\subset\G'(\bbq)$. 

Let $\bf F$ be the Zariski closure of $\varphi(\Gamma)$ in $\G'$. Then the natural map $\varrho:{\bf F}\to\G$ is an $\bbr$-epimorphism. 
Let ${\bf K}=\ker(\varrho)$. Recall that any compact subgroup
of a real algebraic group is itself algebraic. In view of this and since $\varphi(\Gamma)$ has bounded image in ${\bf K}(\bbr)$, 
we get that ${\bf K}(\bbr)$ is compact.

Moreover, $\sigma(\Gamma)$ is bounded in $\G(\gfield_v)$ for all non-Archimedean places and $\Gamma$ is finitely generated,
hence, $\varphi(\Gamma)\cap{\bf F}(\bbz)$ in finite index in $\varphi(\Gamma)$. 
Further, since ${\bf K}(\bbr)$ is compact, $\varrho({\bf F}(\bbz))$ is discrete in $\G(\bbr)$;
we get that $\Gamma$ and $\varrho({\bf F}(\bbz))$ are commensurable. The proof is complete.
\end{proof}


\section{The cocycle and equivariant measurable maps}\label{sec:equiv-map}
In this section we recall a construction due to Margulis which produces 
an equivariant measurable map between certain projective spaces. 
   
Let $\Gamma\subset G$ be a uniform lattice and let $\lf$ be a local field of characteristic zero.
We assume fixed a homomorphism 
\[
\rho:\Gamma\to\PGL_2(\lf)
\] 
whose image is {\em unbounded and Zariski dense} --- in our application, $\rho$ will actually be a monomorphism.


\subsection{Characteristic maps and cocylces}\label{sec:cocycle}
Let $B$ denote the group of upper triangular matrices in $\PGL_2(\lf)$.
We now recall from~\cite[Ch.~V]{Margulis-Book} the construction of a measurable map from 
$Q\backslash G$ to $B\backslash \PGL_2(\lf)$ which is associated to the representation $\rho$. This approach relies on the multiplicative ergodic theorem.   
 
%
%


We fix a fundamental domain $F$ which is the invariant lift of a Dirichlet fundmental domain for $\Gamma$ in $\mathbb H^3$ --- in particular, $F$ is compact and $\partial F$ is a union of finitely many suborbiforlds of lower dimension. 

For any $g\in G$, there exists a unique $\gamma_g\in\Gamma$ so that $g\in F\gamma_g$.
Set
\[
\omega(g):=\rho(\gamma_g).
\]
Then $\omega:G\to\PGL_2(\lf)$ is a Borel map and
\be\label{eq:f-Gamma-equi}
\omega(g\gamma)=\omega(g)\rho(\gamma)\text{ for all $g\in G$ and $\gamma\in\Gamma$.}
\ee
Define $b_\omega'(g,y)=\omega(y)\omega(gy)^{-1}$ for all $g\in G$ and $y\in G$. Note that by~\eqref{eq:f-Gamma-equi}
we have $b_\omega'(g,y)=b_\omega'(g,y\gamma)$ for all $\gamma\in\Gamma$. Define
\[
b_\omega(g,x)=b_\omega'(g,\pi^{-1}(x))\text{ for all $g\in G$ and $x\in X$}.
\]
Then $b_\omega(g_1g_2,x)=b_\omega(g_2,x)b_\omega(g_1,g_2x)$. That is: $b_\omega$ is a cocycle.

Define the cocycle
\be\label{eq:def-cocycle}
u(n,x)=b_\omega(a_{n}, x)\text{ for all $x\in X$ and all $n\in\bbz$};
\ee
where $a_t$ is defined in~\eqref{eq:def-a-t}.

%



\begin{thm}[Cf.~\cite{Margulis-Book}, Ch.~V and VI]\label{thm:meas-map-cocle}\label{thm:equiv-map-bry}
Let the notation be as above.  
\begin{enumerate}
\item There exists a unique $\Gamma$-equivariant Borel measurable map 
\[
\psi: \mathbb S^2\to B\backslash \PGL_2(\lf).
\]

\item There exist some $\lambda_1>0$ so that the following holds. 
Let $w\in \lf^2\setminus\{0\}$. Then for $\vol_X$-a.e.\ $x\in X$ we have  
\[
\lim_{n\to\infty}\frac{1}{n}\log(\|u(n,x)w\|/\|w\|)= \lambda_1
\]
\end{enumerate}
\end{thm}

\begin{proof}
The existence of a measurable map $\psi$ as in~(1) 
is proved in~\cite[Ch.~V, Thm.~3.2]{Margulis-Book}. An alternative approach is~\cite[Ch.~VI, Thm.~4.3]{Margulis-Book}
which is based on the work of Furstenberg.  

The fact that there exists $\lambda_1>0$ so that part~(2) holds follows from the fact that $\rho(\Gamma)$ is unbounded 
and Zariski dense, see~\cite{Fu-Boundary, Golds-Marg} and~\cite[Ch.~V, Thm.~3.2]{Margulis-Book}.  
\end{proof}





\subsection{A set of uniform convergence}\label{sec:set-unif-cov}
Let $\vare>0$. Recall that $F$ is the invariant lift of a compact Dirichlet fundamental domain 
$E$ for the action of $\Gamma$ on $\mathbb H^3$.
Let $\mathcal C_{E,\vare}\subset F$ be as in Lemma~\ref{lem:Egorov-Noatom}.   

Let $e_1=(1, 0)\in \lf^2$. There exists a compact subset $F_{\vare}'\subset \mathcal C_{E,\vare}$ with 
$\vol_X(F_{\vare}')>1-\vare^8$ so that 
the convergence in Theorem~\ref{thm:equiv-map-bry}(2) is uniform. 
That is for every $\eta>0$ there exists some $n_\eta$ so that for all $n>n_{\eta}$
and any $x\in F_\vare'$ we have 
\be\label{eq:unif-conv-F}
\biggl|\Bigl(\frac{1}{n}\log\|u(n,x)e_1\|\Bigr)-\lambda_1\biggr|<\eta\;\;\text{ for all $n>n_\eta$.}
\ee

Since the measure $\vol_X$ is $G$-invariant and in particular, $\SO(2)$-invariant, the following holds. 
There exists a compact subset $F_\vare\subset F_\vare'$ with 
\[
\vol_X(F_\vare)>1-\vare^4
\]
so that for every $g\in F_\vare$ we have $|\{\theta\in[0,2\pi] :r_\theta g\in F_\vare'\}|> 1-\vare^4$ 
where $r_\theta$ denotes the rotation matrix with angle $\theta$.

\medskip

Let ${\tau}>0$ be fixed. For every $\alpha>0$, 
let $F_\vare({\tau},\alpha)\subset F_\vare$ be the subset with the property that for every $g\in F_\vare({\tau},\alpha)$ we have 
\be\label{eq:interior-fd}
|\{\theta\in[0,2\pi]: \mathsf B(a_{{\tau}}r_\theta g, \alpha)\subset F\omega(a_{{\tau}}r_\theta g)\}|>2(1-2\vare^4)2\pi
\ee
where for every $h\in G$ we have $\omega(h)\in\Gamma$ is the element so that $h\in F\omega(h)$ and $\mathsf B(h,\alpha)$ 
denote the ball of radius $\alpha$ centered at $h$.
 
Let ${\tau}$ and $\vare>0$ be fixed. There exists some $\alpha_0=\alpha_0(\vare)$ so that 
\be\label{eq:tau0-alpha0}
\vol(F_\vare({\tau},\alpha_0))>1-2\vare^4.
\ee


\subsection*{Random walks and distances}\label{sec:hyp-dist-rm-wk-ccycl}
Let $\mathcal T$ denote the Bruhat-Tits tree (if $\lf$ is non-Archimedean) or $\bbh^3$ 
if $\lf=\bbc$. Let ${\dt}$ denote the $\PGL_2(\lf)$-invariant metric on $\mathcal T$ --- that is: if $\lf=\bbc$,
then $\dt$ is the hyperbolic metric and if $\lf$ is non-Archimedean, then $\dt$ is the path metric on a tree. 
We fix a base point $o\in \mathcal T$ which we assume to be the image of the identity element of $\PGL_2(\lf)$ in $\mathcal T$.
Abusing the notation, given an element $g\in\PGL_2(\lf)$ we often write $\dt(g,o)$ for $\dt(g\cdot o, o)$.

Recall that $\Gamma$ is a uniform lattice in $G$. Therefore, 
\be\label{eq:rw-trivial-bd}
\dt(u(n,x),o)\leq \ref{k:rw-triv-bd}n
\ee
for some $\consta\label{k:rw-triv-bd}$ depending on the representation $\rho$.

\begin{lemma}\label{lem:rm-wk-ccycl-dist-xi-t}
Let $0<\vare<1/2$ be so that $\ref{k:rw-triv-bd}\vare<0.01\lambda_1$.
Let $F_\vare$ be defined as in~\S\ref{sec:set-unif-cov}.
There exists some $N_0$ so that for any $n>N_0$ we have the following.
\begin{enumerate} 
\item For every geodesic $\xi=\{\xi_t\}\subset\mathcal T$ with $\xi_0=o$ and every $g\in F_\vare$, 
there exists a subset $R_{g,\xi}\subset [0,2\pi]$ with $|R_{g,\xi}|>2(1-\vare)\pi$ so that for all $\theta\in R_{g,\xi}$
we have 
\[
\dt(u(n,r_\theta g)\cdot o,\xi_t)>t+\lambda_1n/3.
\]
\item For every geodesic $\xi=\{\xi_t\}\subset\mathcal T$ with $\xi_0=o$ and every $g\in F_\vare$ we have
\[
\frac{1}{2\pi}\int_0^{2\pi}\dt(u(n,r_\theta g)\cdot o,\xi)\diff\!\theta> t+\lambda_1n/5.
\]
\end{enumerate}
\end{lemma}

\begin{proof}
The proof is a special case of the argument presented in the proof of Lemma~\ref{lem:hyp-2}; 
we repeat parts of the argument for the convenience of the reader.

Apply~\eqref{eq:unif-conv-F} with $\eta=1/2$ and assume $n>n_{1/2}$ for the rest of the discussion.

For any $\theta\in[0,2\pi]$, let $q_\theta=u(n,r_\theta g)\cdot o$. By the hyperbolic law of cosines, see~\eqref{eq:hyp-law-cos}, we have
\[
\cosh(\overline{q_\theta\xi_t})= \cosh(\overline{o\xi_t})\cosh(\overline{oq_\theta})-\cos(\alpha)\sinh(\overline{o\xi_t})\sinh(\overline{oq_\theta})
\]
where $\alpha$ is the angle between $\overline{o\xi_t}$ and $\overline{oq_\theta}$.

Then in view of Lemma~\ref{lem:Egorov-Noatom} and the fact that $g\in F_\vare\subset \mathcal C_{E,\vare}$, 
there exists some $\delta>0$ so that 
\be\label{eq:d-ang-delta-beta}
|\{\theta\in[0,2\pi]: \psi(r_\theta g)\not\in\mathsf N_\delta(\xi_\infty)\}|> 2(1-\vare)\pi
\ee
where $\xi_\infty=\lim_{t\to\infty}\xi_t\in\PGL_2(\lf)/B$.

Therefore, we have 
\be\label{eq:law-cos-Theta-L2-psi}
\dt(u(n,r_\theta g)\cdot o,\xi_t)>t+\tfrac{\lambda_1n}{2}-O_\delta(1) 
\ee
for all $\theta$ which lies in the set appearing in~\eqref{eq:d-ang-delta-beta}. 
Assuming $n$ is large enough, depending on $\delta$, we get part~(1). 

\medskip

Recall that $\dt(u(n,h)\cdot o,o)\leq\ref{k:rw-triv-bd} n$ for all $h\in F$. In particular we have
\be\label{eq:triv-boud-vs-cosine}
\dt(u(n,h)\cdot o,\xi_t)\geq t-\ref{k:rw-triv-bd}n\quad\text{for all $n$.}
\ee

Part~(2) now follows from~\eqref{eq:law-cos-Theta-L2-psi} and~\eqref{eq:triv-boud-vs-cosine}
as Lemma~\ref{lem:hyp-2} was proved using~\eqref{eq:law-cos-Theta-L2} and~\eqref{eq:triv-boud-vs-cosine-Theta}.
\end{proof}

\begin{lemma}\label{eq:close-gen-prog} 
Let $0<\vare<1/2$ with $\ref{k:rw-triv-bd}\vare<0.01\lambda_1$ 
and let ${\tau}>N_0$ be a fixed parameter, where $N_0$ is as in Lemma~\ref{lem:rm-wk-ccycl-dist-xi-t}. 
Suppose $x\in X$ is so that there exists some $g_x\in F_\vare({\tau},\alpha_0)$ with 
\be\label{eq:close-to-N}
\dist(a_{{\tau}}r_\theta g_x,a_{{\tau}}r_\theta x)<\alpha_0/2\quad\text{for all $\theta\in[0,2\pi]$,}
\ee
see~\eqref{eq:tau0-alpha0}. Then the following hold.
\begin{enumerate}
\item Let $\xi=\{\xi_t\}\subset\mathcal T$ with $\xi_0=o$ be a geodesic. There exists a subset $\hat R_{x,\xi}\subset[0,2\pi]$ 
with $|\hat R_{x,\xi}|> 2(1-2\vare)\pi$ so that for all $\theta\in\hat R_{x,\xi}$ we have 
\[
\dt(u({\tau},r_\theta x), \xi_t)> t+ {\tau}\lambda_1/3. 
\]
\item Let $\xi=\{\xi_t\}\subset\mathcal T$ with $\xi_0=o$ be a geodesic. Then
\[
\frac{1}{2\pi}\int_0^{2\pi}\dt(u({\tau},r_\theta x), \xi_t)\diff\!\theta>t+{\tau}\lambda_1/5.
\]
\end{enumerate}
\end{lemma}

\begin{proof}
We first prove part~(1). Let 
\[
R_{g_x}=\{\theta\in[0,2\pi]: \mathsf B(a_{{\tau}}r_\theta g_x, \alpha_0)\subset F\}.
\] 
Then by the definition of $F_\vare({\tau},\alpha_0)$, see~\eqref{eq:interior-fd}, we have $|R_{g_x}|>2(1-2\vare^4)\pi$. 

Let $R_{g_x,\xi}$ be as in Lemma~\ref{lem:rm-wk-ccycl-dist-xi-t}(1) applied to $g_x$ and $\xi$ and put
\[
\hat R_{x,\xi}:=R_{g_x}\cap R_{g_x,\xi}.
\] 
Note that $|\hat R_{x,\xi}|>2(1-2\vare)\pi$.

Let now $\theta\in\hat R_{x,k}$. By~\eqref{eq:close-to-N}, we have 
\[
\dist(a_{{\tau}}r_\theta g_x,a_{{\tau}}r_\theta x)<\alpha_0/2.
\] 
Since $\theta\in R_{x,g_x}$, we have $u({\tau},r_\theta x)=u({\tau},r_\theta g_x)$. 
Moreover, since $\theta\in R_{g_x,\xi}$ and ${\tau}>N_0$, we get from 
Lemma~\ref{lem:rm-wk-ccycl-dist-xi-t}(1) that 
\[
\dt(u({\tau}, r_\theta x),\xi_t)=\dt(u({\tau}, r_\theta g_x),\xi_t)>t+\lambda_1{\tau}/3,
\]
as was claimed in part~(1).

The proof of part~(2) is similar to the proof Lemma~\ref{lem:hyp-2} as we now explicate.
Recall from~\eqref{eq:rw-trivial-bd} that
\be\label{eq:R-theta-or-not}
\dt(u({\tau}, r_\theta x),\xi_t)\geq t-\ref{k:rw-triv-bd}{\tau}
\ee
for any $\theta\in[0,2\pi]$.

Using part~(1) and~\eqref{eq:R-theta-or-not} we obtain the following.
\begin{align*}
\frac{1}{2\pi}\int_0^{2\pi}\dt(u({\tau}, r_\theta x),\xi_t)\diff\!\theta&>(1-2\vare)(t+\tfrac{\lambda_1{\tau}}{3})+ 2(t-\ref{k:rw-triv-bd}{\tau})\vare\\
&> t+\lambda_1{\tau}/4-2\ref{k:rw-triv-bd}\lambda_1\vare {\tau}\\
{}^{\ref{k:rw-triv-bd}\vare<0.01\lambda_1\leadsto}&> t+\lambda_1{\tau}/5.
\end{align*}
The proof if complete.
\end{proof}


\section{The main lemma}\label{sec:basic-lemma}
The following lemma is one of the pivotal ingredients in the proof of Proposition~\ref{thm:psi-circle-circle}, 
and is one of the main technical tools in this paper.

\begin{lemma}[Main Lemma]\label{prop:main-prop-intro}\label{lem:basic-lemma}
Let the notation be as in \S\ref{sec:equiv-map}. Further, assume that
there are infinitely many closed $H$-orbits $\{Hx_i: i\in\bbn\}$ in $X=G/\Gamma$.
There exists some $\lambda_0=\lambda_0(\rho)>0$ with the following property. 


For every $\vare>0$, there exist positive integers $i_0$, ${\tau}$, and $N$ 
with the following properties. For all $i>i_0$, there exists a subset $Z_i\subset Hx_i$ with $\mu_{Hx_i}(Z_i)>1-\vare$ so that 
\[
\dt(u(n{\tau},z), o)> \lambda_0 {\tau}n\quad\text{ for all $z\in Z_i$ and all $n>N$}.
\]
\end{lemma}

The proof of this lemma relies on results in \S\ref{sec:equiv-map}, equidistribution theorems in homogeneous dynamics, and certain maximal inequalities. The proof will occupy the rest of this section.

\medskip

We begin with the following theorem which is a special case of a theorem of Mozes and Shah~\cite{Mozes-Shah} --- 
the proof in~\cite{Mozes-Shah} builds on seminal works on unipotent dynamics by Dani, Margulis, and Ratner.

\begin{thm}\label{thm:equidist}
Let $\Gamma\subset G$ be a lattice and let $\vol_X$ denote the probability $G$-invariant measure on $X=G/\Gamma$.
Assume that there are infinitely many closed $H$-orbits $\{Hx_i: i\in\bbn\}$ in $X$.
For every $i$ let $\mu_{Hx_i}$ denote the $H$-invariant probability measure on $Hx_i$.
Then 
\[
\int f\diff\!\mu_{Hx_i}\to\int f\diff\!\vol\quad\text{for any $f\in C_c(G/\Gamma)$.}
\]
\end{thm}

This theorem plays an important role in the sequel. 
We record a corollary of this theorem here which will be used in \S\ref{sec:proof-of-prop}. 

\begin{cor}\label{cor:equidist}
Let the notation be as in Theorem~\ref{thm:equidist}. 
Let $0<\vare<1/2$ and for each $i$, let $Z_i\subset Hx_i$ be a subset with $\mu_{Hx_i}(Z_i)>1-\vare$.
Let $\delta>0$ and for each $i$, let $\mathcal N_{i,\delta}$ be the open $\delta$-neighborhood of $Z_i$. Then 
there exists some $i_1$ so that
\[
\vol_X(\mathcal N_{i,\delta})>1-3\vare\;\text{ for all $i>i_1$}.
\]
\end{cor}

\begin{proof}
For each $i$ consider a covering of $X\setminus \mathcal N_{i,\delta}$ 
with balls of radius ${\delta}/{2}$ with multiplicity $\kappa$ depending only on $X$. 
The characteristic function of these open sets can be approximated by a precompact family $\mathcal F$ of continuous functions. The claim follows since Theorem~\ref{thm:equidist} holds uniformly on precompact families.
\end{proof}


\subsection{Maximal inequalities}\label{sec:max-ineq}
Let $Y=Hx\subset X$ be a closed $H$-orbit and let $\mu$ be the probability $H$-invariant measure on $Y$.
Let $\mathcal R=[0,2\pi]^{\bbz}$ be equipped with $\diff\!\nu:=\Bigl(\frac{\diff\theta}{2\pi}\Bigr)^{\otimes\bbz}$.

Let $\tau>0$ and define $\eta_Y:\mathcal R\times Y\to\mathcal R\times Y$ by 
\[
\eta_Y((\theta_n),y)=(\eta(\theta_n),a_{\tau}r_{\theta_1} y)
\]
where $\eta:\mathcal R\to\mathcal R$ is the shift map. Then the measure $\nu\times\mu$ is $\eta_Y$-invariant and ergodic.

For any $f\in L^1(Y,\mu)$ the function $1\otimes f\in L^1(\mathcal R\times Y,\nu\times\mu)$.
Therefore, in view of the maximal inequality for $\eta_Y$, there exists an absolute constant $D>0$ so that the following holds. Let $f\in L^1(Y,\mu)$; for any $c>0$ we have
\be\label{eq:max-sl2-inv}
\nu\times\mu\biggl\{(\theta,y)\in \mathcal R\times Y: \sup_n\frac{1}{n}\sum_{\ell=1}^n 1\otimes f(\eta_Y^\ell(\theta,y))\geq c \biggr\}\leq \frac{D\|f\|_1}{c}.
\ee

We also need a maximal inequality similar to and more general than Kolmogorov's inequality in the context of the law of large numbers --- see also~\cite[\S3]{BQ-III}. 

Let $\tau>0$ and define an averaging operator $A_\tau$ on the space of Borel functions on $Y$ by
\[
A_\tau\varphi(y)=\frac{1}{2\pi}\int_0^{2\pi}\varphi(a_\tau r_\theta y)\diff\!\theta.
\]

Consider the space $\mathcal W=Y^{\bbn}$ and let $\omega_y$ 
be the Markov measure associated to $A_{\tau}$ and $y$. That is: for bounded Borel function $\phi_0,\ldots, \phi_m$ 
on $Y$ we have 
\[
\int\phi_1(w_1)\cdots\phi_m(w_m)\diff\!\omega_y(w)=(\phi_1A_\tau(\cdots(\phi_{m-1}A_\tau(\phi_m))\cdots))(y)
\]
where $w=(\cdots,w_{-1},w_1,w_2,\ldots)$.

The main case of interest to us is the trajectories obtained using the operator $A_{\tau}$. 
That is: trajectories of the form 
\be\label{eq:paths-in-W}
\Bigl((w_j)_{j\in\bbz}\Bigr)=\Bigl((w_j)_{j\leq 0},(a_{\tau}r_{\theta_{j}}\cdots a_{\tau}\theta_1w_0)_{j\geq1}\Bigr)
\ee
for a random $(\theta_j)\in \mathcal R$. 
Let $z\in Y$ we let $\mathcal W_z$ be the space of all paths as in~\eqref{eq:paths-in-W} with $w_0=z$. 
In this case, the measure $\omega_z$ is obtained by pushing forward $\nu$ to $\mathcal W_z$.

Fix some $z\in Y$.
Let $\rho:\Gamma\to\PGL_2(\lf)$ be a representation as in Lemma~\ref{lem:basic-lemma}. 
For every $n\geq 1$ define $u_{n,z}:\mathcal R\to\PGL_2(\lf)$ by 
\be\label{def:u-n-w-theta}
u_{z,n}\bigl((\theta_j)\bigr)=u(\tau,r_{\theta_{n}}w_{n-1})\cdots u(\tau,r_{\theta_2}w_1)u(\tau,r_{\theta_1}z);
\ee
where $w_1=a_\tau r_{\theta_1}z$ and for all $j>1$ we have $w_j=a_\tau r_{\theta_j}w_{j-1}$, see~\eqref{eq:paths-in-W}.

Fix some $z\in Y$.
For all $(\theta_j)\in\mathcal R$ and all $n\geq 1$, define $\phi_{\theta,n}:Y\to\bbr$ by
\be\label{def:phi-dt}
\phi_{\theta,n}(y)=\dt\Bigl(u(\tau, y)u_{z,n-1}((\theta_j)),o\Bigr)-\dt\Bigl(u_{z,n-1}((\theta_j)),o\Bigr).
\ee
In view of~\eqref{eq:rw-trivial-bd}, there is some $L=L(\rho, \tau)$, {\em but independent of $Y$}, so that
\be\label{eq:phi-bdd}
|\phi_{\theta,n}|\leq L\;\;\text{ for all $\theta$ and $n$}. 
\ee

Put $\varphi_n\Bigl((\theta_j)\Bigr):=\phi_{\theta,n}(r_{\theta_{n}}w_{n-1})-\frac{1}{2\pi}\int_0^{2\pi}\phi_{\theta,n}(r_\theta w_{n-1})\diff\!\theta$.

\begin{lemma}\label{lem:maximal}
For every $c>0$ and $\delta>0$, there exists some $N_1=N_1(c,\delta,L)$ with the following property.  
Let $z\in Y$ and define $\varphi_\ell$ as above. Then
\[
\nu\biggl(\{\theta=(\theta_j)\in\mathcal R:\max_{n\geq N_1}\frac{1}{n}|\textstyle\sum_{\ell=1}^n \varphi_\ell(\theta)|> c\}\biggr)\leq\delta.
\]
\end{lemma}

\begin{proof}This lemma is proved using the following maximal inequality which follows, e.g.\ by combining~\cite[p.~386]{Loeve} with~\cite[Thm~1.1]{Fa-Kl}, see also~\cite{Haj-Ren, Chow}. 

\medskip

{\em Let $(\Omega,\mathcal B, \beta)$ be a standard probability space and let $\{\zeta_n\}$ 
be a sequence of bounded Borel functions on $\Omega$ so that $\mathbb E_\beta(\zeta_n|\zeta_{n-1},\ldots,\zeta_1)=0$ for every $n$.
Then for every $N_1\geq 1$ and every $c>0$ we have 
\begin{multline}\label{eq:max-ineq-martingail}
\beta\Bigl(\{\omega\in\Omega:\max_{N_1\leq n\leq N}\frac{1}{n}|\textstyle\sum_{\ell=1}^n\zeta_\ell(\omega)|> c\}\Bigr)\leq\\
\frac{1}{c^2}\biggl(\textstyle\sum_{n=N_1}^{N}\frac{\int\zeta_n^2}{n^2}+\frac{1}{N_1^2}\textstyle\sum_{n=1}^{N_1}{\int\zeta_n^2}\biggr).
\end{multline}
}

Returning to our setup, we now observe that
\[
\mathbb E_{\nu}(\varphi_n|\varphi_{n-1},\ldots,\varphi_1)=\mathbb E_{\nu}\Bigl(\mathbb E_{\nu}(\varphi_n|\theta_{n-1},\ldots,\theta_1)|\varphi_{n-1},\ldots,\varphi_1\Bigr).
\]
Moreover, we have $\mathbb E_{\nu}(\varphi_n|\theta_{n-1},\ldots,\theta_1)=0$. Hence 
\[
\mathbb E_{\nu}(\varphi_n|\varphi_{n-1},\ldots,\varphi_1)=0.
\]

Therefore, we may apply~\eqref{eq:max-ineq-martingail} with the space $(\mathcal R, \mathcal B^{\otimes \bbz}, \nu)$ 
and the sequence $\{\varphi_n\}$ of functions.
Since $\int\varphi_n^2\leq 2L^2$, see~\eqref{eq:phi-bdd}, and  $\sum\frac{1}{n^2}$ is a convergent series,  the lemma follows. 
\end{proof}

%

\subsection{Conclusion of the proof of Lemma~\ref{lem:basic-lemma}}\label{sec:proof-main-lemma}
%
%
%
%
%
%


Let $0<\vare<1/2$ with $\ref{k:rw-triv-bd}\vare<0.01\lambda_1$, see~\eqref{eq:rw-trivial-bd}, 
and let $\tau=N_0+1$ where $N_0$ is as in Lemma~\ref{lem:rm-wk-ccycl-dist-xi-t}. 

Let $\delta>0$ be small enough so that for any $g,g'\in G$ with $g=hg'$ and $\|h-I\|\leq \delta$, we have
\be\label{eq:choose-delta}
\dist(a_{\ell}r_\theta g,a_{\ell}r_\theta g')<\alpha_0/2
\ee
for all $\theta\in[0,2\pi]$ and all $0<\ell<2\tau$, see~\eqref{eq:tau0-alpha0}.

Let $\mathcal N$ denote the $\delta$-neighborhood of $F_\vare(\tau,\alpha_0)$ in $X$.
In view of~\eqref{eq:choose-delta} and Lemma~\ref{eq:close-gen-prog}(2) we have the following. 
Let $\xi=\{\xi_t\}\subset\mathcal T$ with $\xi_0=o$ be a geodesic. Then
\be\label{eq:av-generic-cirle-pf}
\frac{1}{2\pi}\int_0^{2\pi}\dt(u({\tau},r_\theta x), \xi_t)\diff\!\theta>t+{\tau}\lambda_1/5
\ee
for all $x\in\mathcal N$. Let us write $\lambda_2:=\lambda_1/5$.

%
%

\medskip

By Theorem~\ref{thm:equidist}, there exists some $i_0$ so that for all $i>i_0$ we have 
\[
\mu_{Hx_i}(\mathcal N\cap Y_i)\geq 1-2\vare^4.
\]
Fix some $i>i_0$ and let $Y=Y_i$ and $\mu=\mu_{Hx_i}$.

By Fubini's theorem, there exists a compact subset $Y'\subset \mathcal N\cap Y$ with $\mu(Y')\geq1-\vare^3$
so that for all $y\in Y'$, we have 
\[
|\{\theta\in[0,2\pi]: r_\theta y\in\mathcal N\cap Y\}|\geq1-\vare.
\]

Apply the maximal inequality in~\eqref{eq:max-sl2-inv} with $Y$, $f=\mathbbm 1_{Y\setminus Y'}$, and $\tau$ as above. 
In consequence, there exists an absolute constant $D>0$ so that 
\be\label{eq:max-sl2-inv-used}
\nu\times\mu\biggl\{(\theta,y)\in \mathcal R\times Y: \sup_n\frac{1}{n}\sum_{\ell=1}^n 1\otimes f(\eta^\ell(\theta,y))\geq \vare \biggr\}\leq D\vare^2.
\ee
Therefore, if $\vare$ is small enough, Fubini's theorem implies the following: 
there exists some $Z'\subset Y$ with $\mu(Z')>1-2\vare$ so that for every $z\in Z'$ we have 
\be\label{eq:lem-max-used-1}
\nu\biggl(\{\theta\in\mathcal R:\sup_n\frac{1}{n}\sum_{\ell=1}^n 1\otimes f(\eta_Y^\ell(\theta,z))\geq \vare\}\biggr)\leq \vare.
\ee
Let $\theta=(\theta_j)$ be in the complement of the set on the left side of~\eqref{eq:lem-max-used-1} and let $(w_j)\in \mathcal W_z$ 
be the path obtain from this $\theta$, see~\eqref{eq:paths-in-W}. Then
\be\label{eq:max-ineq-used-2}
\frac{1}{n}\sum_{\ell=1}^n \mathbbm 1_{Y'}(w_\ell)\geq 1-\vare.
\ee

Put $I_n(\theta):=\{1\leq \ell\leq n: w_\ell\in Y'\}$ and $I_n'(\theta):=\{1\leq\ell\leq n\}\setminus I_n(\theta)$.
Then, by~\eqref{eq:max-ineq-used-2}, we have   
\be\label{eq:max-ineq-mi-use}
\# I_n(\theta)\geq (1-\vare)n.
\ee

Apply now Lemma~\ref{lem:maximal} with this $z\in Z'$ and with $c=0.1\lambda_2$ and $\delta=\vare$.
Let $N_1$ be as in Lemma~\ref{lem:maximal} for these choices. Then
\be\label{eq:lem-max-used}
\nu\biggl(\{\theta\in\mathcal R:\max_{n\geq N_1}\frac{1}{n}|\textstyle\sum_{\ell=1}^n \varphi_\ell(\theta)|> 0.1\lambda_2\}\biggr)\leq \vare.
\ee

Let $\mathcal R_z\subset\mathcal R$ 
be the compliment of the union of sets appearing on the left sides of~\eqref{eq:lem-max-used-1} and~\eqref{eq:lem-max-used}.
Let $\theta\in\mathcal R_z$ and write $I_n$ and $I_n'$ for $I_n(\theta)$ and $I'_n(\theta)$, respectively. 
Let $\ell\in I_n$. Then 
\[
w_{\ell}=a_{\tau}\theta_{\ell-1}\cdots a_{\tau}r_{\theta_1}z\in Y',
\] 
and by the definition of $Y'$, $r_{\beta}w_{\ell}\in Y\cap\mathcal N$ for some $\beta\in[0,2\pi]$.

Apply~\eqref{eq:av-generic-cirle-pf} with $g=r_{\beta}w_{\ell}$ and the geodesic segment 
$\xi$ connecting $o$ to $q:=u_{\ell}((\theta_j,w_j))$, see~\eqref{def:u-n-w-theta}. 
Let us put $t=\dt(q, o)$ and parametrize so that $\xi(t)=o$. In consequence, we have the following: 
\begin{align}
\label{eq:traj-large-genric-set}t+\lambda_2\tau&<\frac{1}{2\pi}\int_0^{2\pi}\dt(u(\tau,r_\theta r_\beta w_{\ell})q, \xi(t))\diff\!\theta\\
\notag&=\frac{1}{2\pi}\int_0^{2\pi}\dt(u(\tau,r_\theta w_{\ell})q, o)\diff\!\theta&&o=\xi(t).
\end{align}

Moreover, by~\eqref{def:phi-dt} we have
\[
\frac{1}{2\pi}\int_0^{2\pi}\phi_{\theta,\ell+1}(r_\theta w_{\ell})\diff\!\theta=\frac{1}{2\pi}\int_0^{2\pi}\dt(u(\tau,r_\theta w_{\ell})q,o)-\dt(q,o)\diff\!\theta.
\]
Recall that $\dt(q,o)=t$, therefore, using the above and~\eqref{eq:traj-large-genric-set} we get that 
\be\label{eq:average-big-Y}
\frac{1}{2\pi}\int_0^{2\pi}\phi_{\theta,\ell+1}(r_\theta w_{\ell})\diff\!\theta\geq\lambda_2\tau\quad\text{if $\ell\in I_n$.}
\ee
Recall further from~\eqref{eq:rw-trivial-bd} that $\dt(u(\tau,g),e)\leq \ref{k:rw-triv-bd}\tau$ for all $g\in F$; thus, 
we may use the triangle inequality and get also the trivial estimate 
\begin{align}
\notag\biggl|\frac{1}{2\pi}\int_0^{2\pi}\phi_{\theta,\ell+1}(r_\theta w_{\ell})\diff\!\theta\biggr|&\leq\frac{1}{2\pi}\int_0^{2\pi}|\dt(u(\tau,r_\theta w_{\ell})q,o)-\dt(q,o)|\diff\!\theta\\
\label{eq:average-trivial-est}&\leq \ref{k:rw-triv-bd}\tau
\end{align}
for all $0\leq \ell\leq n-1$.

In view of~\eqref{eq:average-big-Y} and~\eqref{eq:average-trivial-est}, for every $n\in\bbn$ we have
\begin{align}
\notag\sum_{\ell=1}^{n}\frac{1}{2\pi}\int_0^{2\pi}\phi_{w,\ell}(r_\theta w_{\ell-1})\diff\!\theta&\geq (\#I_n)\lambda_2\tau-(n-\#I_n)\ref{k:rw-triv-bd}\tau\\
\notag{}^{\text{\eqref{eq:max-ineq-mi-use}$\leadsto$}}&\geq (1-\vare)\lambda_2\tau n-\vare \ref{k:rw-triv-bd}\tau n\\
\label{eq:average-est-main}{}^{\vare<0.1\;\&\;\vare \ref{k:rw-triv-bd}<0.1\lambda_2\leadsto}&\geq\lambda_2\tau n/2
\end{align}

Let now $n> N_1$. 
Therefore, since $((\theta_j,w_j))\in\mathcal R_z$, we conclude from~\eqref{eq:lem-max-used}
that
\be\label{eq:lem-max-used-again}
\frac{1}{n}\biggl|\sum_{\ell=1}^n \varphi_{\ell}(\theta)\biggr|\leq 0.1\lambda_2.
\ee
Recall again from~\eqref{def:phi-dt} the definition of $\phi_{\theta,n}(r_{\theta_{n}}w_{n-1})$,
also recall that $\varphi_n=\phi_{\theta,n}-\frac{1}{2\pi}\int_0^{2\pi}\phi_{\theta,n}$. We thus obtain
\[
\sum_{\ell=1}^n \varphi_{\ell}(\theta)=\dt\Bigl(u_{z,n}(\theta),o\Bigr)-
\sum_{\ell=1}^n\frac{1}{2\pi}\int_0^{2\pi}\phi_{\theta,\ell}(r_\theta w_{\ell-1})\diff\!\theta.
\]
This, together with~\eqref{eq:lem-max-used-again} and~\eqref{eq:average-est-main}, implies that for all 
$n\geq N_1$ we have
\be\label{eq:final-est-rand-trj}
\dt\Bigl(u_{z,n}(\theta),o\Bigr)\geq (\tau/2-1/10)\lambda_2n\geq \lambda_2\tau n/3=\lambda_1\tau n/15.
\ee

To get Lemma~\ref{lem:basic-lemma} from~\eqref{eq:final-est-rand-trj} it remains to note that 
trajectories 
\[
\{a_{\tau}\theta_{n}\cdots a_{\tau}r_{\theta_1}z: (\theta_j)\in\mathcal R\}
\] 
give rise to the rotation invariant distribution on the boundary circle corresponding to $Hg_z$, recall that $g_z\in F$. 
Moreover, for $\nu$-a.e.\ $\theta=(\theta_j)\in\mathcal R$ there exists a unique geodesic $\{\xi_{\theta,t}\}$ with $\xi_{\theta,0}=g_z$ 
so that the trajectory $a_{\tau}\theta_{n}\cdots a_{\tau}r_{\theta_1}g_z$ is at a sublinear distance from $\{\xi_{\theta,t}\}$, see Lemma~\ref{lem:hyp-1}. Indeed even a central limit theorem holds for these trajectories~\cite{Furst-Pois, Furst-Kest, Guivarch, Berger-CLT}. 
\qed

\section{Proof of Proposition~\ref{thm:psi-circle-circle}}\label{sec:proof-of-prop}\label{sec:proof-prop-final}

In this section we complete the proof of Proposition~\ref{thm:psi-circle-circle}. 
The proof uses Lemma~\ref{lem:basic-lemma}; we begin with some preliminary statements. 

Let the notation be as in Proposition~\ref{thm:psi-circle-circle}.
In particular, recall that for any $g\in\PGL_2(\lf_v)$, 
$\tc_g$ denotes  the graph of the linear fractional transformation 
\[
g:\mathbb P\lf_v\to\mathbb P\lf_v.
\]
Alternatively, $\tc_g$ is the image of $\{(h,ghg^{-1}): h\in\PGL_2(\lf_v)\}$ in $\mathbb P\lf_v\times\mathbb P\lf_v$.

Let $\alpha,\beta\in\lf_v$ and define 
\[
{\rm Crs}(\alpha,\beta)=\{([\alpha,1],[s,1])\in\mathbb P\lf_v\times\mathbb P\lf_v\}\cup\{([r,1],[\beta,1])\in\mathbb P\lf_v\times\mathbb P\lf_v\}.
\]
We refer to ${\rm Crs}(\alpha,\beta)$ as {\em crosses}.

\begin{lemma}\label{lem:circles-crosses}
Let $\{g_{i}\}$ be a squence of elements in $\PGL_2(\lf_v)$. 
Then at least one of the following holds.
\begin{enumerate}
\item There exists a subsequence $\{g_{i_m}\}$ and some $g\in\PGL_2(\lf_v)$ so that $\tc_{g_{i_m}}\to\tc_g$, or
\item there exists a subsequence $\{g_{i_m}\}$ so that $\{\tc_{g_{i_m}}\}$ converges to a cross or a union of a line and a point.
\end{enumerate}
\end{lemma}

\begin{proof}
If there exists a subsequence $\{g_{i_m}\}$ and some $g\in\PGL_2(\lf_v)$
so that $g_{i_m}\to g$, then $\tc_{g_{i_m}}\to\tc_g$. That is: part~(1) holds.

Therefore, we may assume that $g_i\to\infty$ and will show that part~(2) holds in this case. 
Passing to a subsequence we may assume all $g_i$'s are in one $\PSL_2(\lf_v)$ coset, hence we 
assume $\det(g_i)=k$ for all $i$.

We use the projective coordinates, in those coordinates we have: $\mathbb P\lf_v=\{[r,s]: r,s\in\lf_v\}$ and 
\[
\tc_{g_i}=\{([r,s],[a_ir+b_is, c_ir+d_is]\}
\]
where $g_i=\begin{pmatrix}a_i & b_i\\ c_i & d_i\end{pmatrix}$. 

First let us assume that for all but finally many $g_i$'s we have $c_i\neq0$.
Omitting these finitely many terms, we assume $c_i\neq 0$. Using the non-homogeneous coordinates, we have
\begin{multline*}
\tc_{g_i}=\Bigl\{([r,1],[\tfrac{a_ir+b_i}{c_ir+d_i}, 1]: -d_i/c_i\neq r\in \lf_v\Bigr\}\quad\cup \\ \{([1,0], [a_i/c_i, 1]\}\cup\{([-d_i/c_i,1],[1,0])\}
\end{multline*}

Alternatively, except for two points, the graph $\tc_{g_i}$ is given by the equation 
\be\label{eq:graph-prod}
(s-a_i/c_i)(r+d_i/c_i)=-k/c_i^2
\ee
--- recall that $\det(g_i)=k$. The missing two points can be obtained by taking limit as $r\to\infty$ or $s\to\infty$.   

Now if the sequence $\{c_i\}$ is bounded away from $\infty$ and $0$, then since $g_i\to\infty$ we get from~\eqref{eq:graph-prod} 
that $\tc_{g_i}$ converges to $\{([1,0], [-d_i/c_i,1]\}$ or $\{([a_i/c_i,1],[1,0])\}$. 
Therefore, we may assume that passing to a subsequence either $c_i\to0$ or $c_i\to\infty$. 
In either case we get the conclusion in part~(2).
 
It remains to consider the case where $c_{i_m}=0$ along a subsequence. 
In this case, $\{\tc_{g_{i_m}}\}$ converges to the union of a line and a point. 
\end{proof}

For $j=1,2$, let 
\[
\mathfrak p_j:\PGL_2\times\PGL_2\to\PGL_2
\] 
be the projection onto the $j$-th component; put 
$\sigma_j=\mathfrak p_j\circ\sigma:\Gamma\to\PGL_2(\lf_v)$.

If $\sigma(\Gamma)\subset\sG(\gfd_v)$ is unbounded, then either $\sigma_1(\Gamma)$
or $\sigma_2(\Gamma)$ is unbounded. Using Lemma~\ref{prop:main-prop-intro} we can show that indeed both these projections are unbounded:

\begin{lemma}\label{prop:sigma-Delta-unbd}\label{cor:proj-1-2-unbd}
Let $M=\bbh^3/\Gamma$ be a finite volume 
hyperbolic $3$-manifold which contains infinitely many totally geodesic surfaces, 
$\{\sfc_i:i\in\bbn\}$.  
For every $i\geq 1$ let $\Delta_i\subset \Gamma$ denote the fundamental group of the surface $\sfc_i$.
Suppose $\sigma_j(\Gamma)$ is unbounded for some $j=1,2$.
Then both $\sigma_1(\Delta_i)$ and $\sigma_2(\Delta_i)$ are unbounded for all large enough $i$. In particular, both $\sigma_1(\Gamma)$ and $\sigma_2(\Gamma)$ are unbounded.
\end{lemma}

\begin{proof}
For every $i$ let $\sH_i$ denote the Zariski closure of $\sigma(\Delta_i)$.
Let $j$ be so that $\sigma_j(\Gamma)$ is unbounded.
Then it follows from Lemma~\ref{prop:main-prop-intro} that $\sigma_j(\Delta_i)$ is unbounded 
for all large enough $i$. Therefore, by Lemma~\ref{lem:proj} we have:
for all large enough $i$, there exists some $g_i\in\PGL_2(\lf_v)$ so that
\[
\sH_i(\lf_v)\cap\{(h,g_ihg_i^{-1}):h\in\PGL_2(\lf_v)\}
\]
is a subgroup of index at most $8$ in $\sH_i(\lf_v)$.

The claims hold for all such $i$.
\end{proof}


\subsection{The definition of $\Psi$}\label{sec:def-Psi}
Recall that $M=\bbh^3/\Gamma$ is a closed 
hyperbolic $3$-manifold containing infinitely many totally geodesic surfaces, 
$\{\sfc_i:i\in\bbn\}$. 

For each $i\in\bbn$, we let $\Delta_i\subset\Gamma$ be the fundamental group of $\sfc_i$. 
Let $\sH_i\subset\PGL_2\times\PGL_2$ denote the Zariski closure of $\sigma(\Delta_i)$ for every $i$.

Let $i$ be large enough so that Lemma~\ref{prop:sigma-Delta-unbd} holds true. 
In particular, there exists some $g_{v,i}\in\PGL_2(\lf_v)$ so that
\[
\{(h,g_{v,i}hg_{v,i}^{-1}): h\in\PSL_2(\lf_v)\}
\]
has index at most $8$ in $\sH_i(\lf_v)$. 
Let $\tc_i$ denote the image of $\sH_i(\lf_v)$ in $\mathbb P\lf_v\times\mathbb P\lf_v$; as was mentioned before, $\tc_i$ 
equals the image of 
\[
\{(h,g_{v,i}hg_{v,i}^{-1}): h\in\PGL_2(\lf_v)\}
\] 
in $\mathbb P\lf_v\times\mathbb P\lf_v$.

Recall that each $\sfc_i$ gives rise to a periodic $H$-orbits, $Hx_i$.
For every $i$, the orbit $Hx_i$ corresponds to a closed $\Gamma$-orbit 
\[
\mathcal C_i=\{C_i\Gamma\}\subset\mathcal C.
\]

Identify $X$ with $F$, then $Hx_i$ is identified with a subset $Y_i\subset F$; note that $Y_i$ has only finitely many connected components. For every $y\in Hx_i$ we let $g_y\in F$ be the point corresponding to $y$. 
The orbit $Hg_y$ gives rise to a plane $P_y$ in $\bbh^3$ and
a circle 
\be\label{eq:C-y-gamma-i}
\text{$C_y:=\partial P_y=C_i\gamma_y\in\mathcal C_i$ for some $\gamma\in\Gamma$.}
\ee

By Lemma~\ref{cor:proj-1-2-unbd}, $\sigma_j(\Gamma)$ is unbounded for $j=1,2$.
For $j=1,2$, let $u_j$ be the cocycle and $\psi_j$ the equivariant map constructed in \S\ref{sec:cocycle} 
using the representation $\sigma_j$. 
Define
\be\label{eq:def-Psi-final}
\Psi:=(\psi_1,\psi_2).
\ee
We will show that $\Psi$ satisfies the claims in Proposition~\ref{thm:psi-circle-circle}.

\begin{lemma}\label{lem:Psi-tgC-tc}
Let $\vare>0$. There exist 
\begin{itemize}
\item natural numbers $N=N(\vare)$ and $\tau=\tau(\vare)$,
\item for every $i>N$, a subset $Z'_i\subset Hx_i$ with $\mu_{Hx_i}(Z'_i)>1-\vare$, and 
\item for every $z\in Z'_i$, a subset $R_z\subset[0,2\pi]$ with $|R_z|>2(1-\vare)\pi$ 
\end{itemize}
so that the following hold. 
\begin{enumerate}
\item For every $z\in Z'_i$, $\theta\in R_z$, $n>N_0$, and $j=1,2$ we have 
\be\label{main-lem-use}
\dt(u_j(n\tau,r_\theta z), o)> \lambda_0\tau n.
\ee
\item For every $z\in Z'_i$ and $\theta\in R_z$ put
\[
\beta_{g_z,\theta}:=\lim_n a_{n\tau}r_\theta g_z\in C_z=C_i\gamma_z
\] 
where $g_z$, $C_z$, and $\gamma_z$ are as in~\eqref{eq:C-y-gamma-i}. Then
\be\label{eq:tot-geod-circ-to-circ}
\Psi(\beta_{g_z,\theta})\in\tc_i\sigma(\gamma_z).
\ee
\end{enumerate}
\end{lemma}

\begin{proof}
By Lemma~\ref{cor:proj-1-2-unbd}, $\sigma_j(\Gamma)$ is unbounded for $j=1,2$.
Let 
\[
\text{$i_0=i_0(\vare^2/2)$, $\tau=\tau(\vare^2/2)$, and $N=N(\vare^2/2)$}
\] 
be so that the conclusion of Lemma~\ref{prop:main-prop-intro} holds with 
$\vare^2/2$ and both $\sigma_1$ and $\sigma_2$. 
Moreover, assume that  Lemma~\ref{prop:sigma-Delta-unbd} holds true for all $i>i_0$.

Let $Z_i\subset Y_i$ be a subset with $\mu_{Hx_i}(Z_i)>1-\vare^2$ so that 
the conclusion of Lemma~\ref{prop:main-prop-intro} holds for all $z\in Z_i$ and for both $\sigma_1$ and $\sigma_2$. 
   
For every $z\in Z_i$, let 
\[
R_z=\{\theta\in[0,2\pi]: r_\theta z\in Z_i\}.
\]
By Fubini's theorem there is a subset $Z'_i\subset Z_i$ with $\mu_{Hx_i}(Z_i')>1-\vare$
so that for all $z\in Z_{i}'$ we have $|R_z|>2(1-\vare)\pi$.

We will show that the lemma holds with this $N$, $Z'_i$, and $R_z$.
Let $z\in Z_i'$. Then by Lemma~\ref{prop:main-prop-intro} we have
\be\label{main-lem-use-again}
\dt(u_j(n\tau,r_\theta z), o)> \lambda_0\tau n,\quad\text{ $\theta\in R_z$, $n>N$, and $j=1,2$}.
\ee
This establishes the part~(1).


Let $\theta\in R_z$. Then by~\eqref{main-lem-use} and Lemma~\ref{lem:hyp-1} we have the following.
There exists a unique geodesic $\{\xi_t:t\in\bbr\}\subset\mathcal T$ with $\xi_0=o$ so that 
\be\label{eq:uj-sublinear-xi-limit}
u_j(n\tau, r_\theta z)\to\xi_\infty\in\partial\mathcal T.
\ee
Moreover, for every $\eta>0$ there is some $N'=N'(\ref{k:rw-triv-bd},\lambda_0, \tau_0, N_0,\eta)$ so that 
\be\label{eq:uj-sublinear-xi}
\dt(u_j(n\tau, r_\theta z),\{\xi_t\})\leq \eta n\tau\quad\text{for all $n>N'$ and $j=1,2$.}
\ee

Note also that $a_tr_\theta\in H$ for all $t$ and $\theta$. Therefore, $\{a_{n\tau_0}r_\theta g_z:n\in\bbz\}\subset P_z$. 
Moreover, since $Hg_z\Gamma$ is closed, we get the following:
\be\label{eq:in-subtree}
\text{$u_j(n_m\tau,r_\theta z)\in\mathcal T'$ for a sequence $n_m\to\infty$}
\ee
where $\mathcal T'$ is the subtree corresponding to $\sigma(\gamma)^{-1}\H_i\sigma(\gamma)$.

Since by~\eqref{eq:uj-sublinear-xi-limit} we have $u_j(n_m\tau,r_\theta z)\to\xi_\infty$, the definition of $\Psi$,~\eqref{eq:def-Psi-final},
and~\eqref{eq:in-subtree} imply that $\Psi(\beta_{g_z,\theta})\in\partial\mathcal T'=\tc_i\sigma(\gamma_z)$. 

This finishes the proof of part~(2) and the lemma. 
\end{proof}

Recall that for a circle $C$ we denote the length measure on $C$
by $\Gm_C$. The following is a crucial step in the proof Proposition~\ref{thm:psi-circle-circle}. 

\begin{lemma}\label{lem:F'-vare-C'-vare}
For every $\vare>0$ there exists a subset $\hat F_\vare\subset F$ with 
\[
\vol_X(\hat F_\vare)>1-4\vare
\]
and for every $g\in \hat F_\vare$ a subset $\hat C_{g,\vare}\subset C_g$ with $\Gm_{C_g}(\hat C_{g,\vare})\geq (1-2\vare)\Gm_{C_g}(C_g)$
so that one of the the following holds. 
\begin{enumerate}
\item There exists some $h_{g}\in\PGL_2(\lf_v)$ so that $\Psi(\hat C_{g,\vare})\subset\tc_{h_{g}}$, or
\item $\Psi(\hat C_{g,\vare})$ is contained in a cross or the union of a line and a point. 
\end{enumerate}
\end{lemma}

\begin{proof}
Fix some $\vare>0$. Let $i_0, N, \tau$, $Z'_i$, and $R_z$ be as in Lemma~\ref{lem:Psi-tgC-tc} applied with this $\vare$.

Let $F_{\vare}$ and $N_0$ be as in Lemma~\ref{lem:rm-wk-ccycl-dist-xi-t}.
Fix a geodesic $\{\xi_t\}$ emanating  from $o$ for the rest of the proof. Let $R_{g}\subset[0,2\pi]$ be 
the set where Lemma~\ref{lem:rm-wk-ccycl-dist-xi-t} holds for $\{\xi_t\}$, $g$, $u_1$, and $u_2$. In particular,
$|R_{g}|>2(1-\vare)\pi$ and for $j=1,2$ we have
\be\label{eq:enif-psi-bigset}
\dt(u_j(n,r_\theta g), o)> \lambda_1 n/3\;\;\;\text{ for all $\theta\in R_{g}$, $n>N_0$, and $t\in\bbr$}.
\ee


For any $g\in F_\vare$, let $P_g$ be the plane corresponding to $Hg$; put $C_g=\partial P_g$. 
For any $\theta\in R_g$, define 
\[
\beta_{g,\theta}:=\lim_n a_{n\tau}r_\theta g\in C_g.
\] 
Then $\Psi(\beta_{g,\theta})\in \Psi(C_g)$.

Given $i$ and $\delta>0$, let $\mathcal N_{i,\delta}$ be the $\delta$-neighborhood of $Z_i'$.
By Corollary~\ref{cor:equidist}, there exists some $i_1(\delta)$ so that $\vol_X(\mathcal N_{i,\delta})>1-3\vare$ for all $i>i_1(\delta)$.

\begin{claim} 
For every $\eta>0$, there exists some $\delta>0$ so that the following holds.
Let $i> i_\eta:=\max\{i_0, i_1(\delta)\}$ and let $g\in F_\vare\cap \mathcal N_{i,\delta}$. 
Then there exists some $C'_{g,\eta}\subset C_g$ with 
\[
\Gm_{C_g}(C'_{g,\eta})>(1-3\vare)\Gm_{C_g}(C_g)
\]
so that $\Psi(C'_{g,\eta})$ lies in the $\eta$-neighborhood of $\tc_{h}$ for some $h\in\PGL_2(\lf_v)$.
\end{claim}


Let us first assume the claim and finish the proof of the lemma.
For every $m\in\bbn$, let $\eta_m=1/m$ and let $\delta_m$ and $i_m$ be given by applying the claim with $\eta_m$. Set 
\[
\hat F_\vare:=F_\vare\cap \biggl(\bigcap_{\ell\geq 1}{\bigcup_{m\geq\ell} \mathcal N_{i_m,\delta_m}}\biggr);
\] 
note that $\vol_X(\hat F_\vare)\geq 1-4\vare$. 

Let $g\in \hat F_\vare$ and let $C_g$ be the corresponding circle. 
Then there exists a subsequence $\{m_k\}$ so that $g\in \mathcal N_{i_{m_k},\delta_{m_k}}$.
In view of the claim, for every $k$, there exists some $C'_{g,m_k}\subset C_g$ and some $h_{k}\in\PGL_2(\lf_v)$ with
$
\Gm_{C_g}(C'_{g,m_k})>(1-2\vare)\Gm_{C_g}(C_g)
$
so that 
\be\label{eq:apply-claim}
\text{$\Psi(C'_{g,m_k})$ lies in the $1/m_k$-neighborhood of $\tc_{h_{k}}$.}
\ee

Apply Lemma~\ref{lem:circles-crosses} with the sequence $\{\tc_{h_{k}}:k\in\bbn\}$. 
Then there exists a subsequence $\{k_i\}$ so that one of the following holds.
\begin{enumerate}
\item $\lim_i\tc_{h_{k_i}}=\tc_h$ for some $h\in\PGL_2(\lf_v)$, or
\item $\lim_i\tc_{h_{k_i}}$ is contained in a cross or the union of a line and a point.
\end{enumerate}

Let 
\[
\hat C_{g,\vare}:=\bigcap_{\ell\geq 1}\bigcup_{i\geq\ell}  C'_{g_{k_i},{m_{k_i}}},
\] 
then $\Gm_C(\hat C_{g,\vare})>(1-2\vare)\Gm(C_g)$. Moreover, for every $\beta\in \hat C_{g,\vare}$ we have 
$\Psi(\beta)\in \lim_i\tc_{h_{k_i}}$. The lemma follows.
\end{proof}

\begin{proof}[Proof of the claim]
Let $g\in F_\vare$. Recall that we fixed a geodesic $\{\xi_t\}$ emanating from $o$, and let $R_{g}\subset[0,2\pi]$ be 
the set where Lemma~\ref{lem:rm-wk-ccycl-dist-xi-t} holds for $\{\xi_t\}$, $g$, $u_1$, and $u_2$.
For $g\in F_{\vare}$ and $\theta\in R_g$, define 
\[
u(n\tau,r_\theta g)=(u_1(n\tau,r_\theta g),u_2(n\tau,r_\theta g)).
\]

For any $z\in Z'_i$, let $R_z$ be as in Lemma~\ref{lem:Psi-tgC-tc}. Similarly, for all $i>i_0$, $z\in Z_i'$, and $\theta\in R_z$ define
\[
u(n\tau, r_\theta g_z)=\Bigl(u_1(n\tau,r_\theta g_z), u_2(n\tau,r_\theta g_z)\Bigr).
\]

Thanks to~\eqref{main-lem-use},~\eqref{eq:enif-psi-bigset}, and Lemma~\ref{lem:hyp-1},
there exists some $n_\eta$ so that for all $n\geq n_\eta$ we have $u(n\tau, r_\theta g_z)$ and $u(n\tau,r_\theta g)$ approximate 
$\Psi(\beta_{g_z,\theta})$ and $\Psi(\beta_{g,\theta})$, respectively, within $\eta/4$.

Let $n=n_\eta$ and let $\delta>0$ be so that if ${\rm d}(h_1,h_2)<\delta$ 
for $h_1,h_2\in G$, then ${\rm d}(a_{n\tau}h_1, a_{n\tau}h_2)\leq \eta/4$ 
where ${\rm d}$ denotes the right invariant Riemannian metric on $G$.

Apply Corollary~\ref{cor:equidist} with the sets $\mathcal N_{i,\delta}$ and let $i_1(\delta)$ be as in that Corollary. 
Let $i>\max\{i_0, i_1(\delta)\}$ and let $g\in F_\vare\cap\mathcal N_{i,\delta}$. 
Then there exists some $g_z\in Z_i'$ so that ${\rm d}(g,g_z)<\delta$. 

The claim thus holds with $C'_{g,\eta}=\{\beta_{g,\theta}:\theta\in R_z\cap R_g\}$.  
\end{proof}




\begin{proof}[Proof of Proposition~\ref{thm:psi-circle-circle}]

First note that in view of Lemma~\ref{lem:PsiGmZd}, $\Psi$ satisfies part~(1) in the proposition --- recall that $\PGL_2\times\PGL_2$ acts transitively on $\mathbb P\lf_v\times\mathbb P\lf_v$, 

We now show that $\Psi$ also satisfies part~(2) in the proposition. 
We claim that there exists a full measure subset $\hat F\subset F$, and for every $g\in \hat F$ 
a full measure subset $\hat C_g\subset C_g$ so that one of the following holds. 
\begin{enumerate}
\item There exists some $h_{g}\in\PGL_2(\lf_v)$ so that $\Psi(\hat C_{g})\subset\tc_{h_{g}}$, or
\item $\Psi(\hat C_{g})$ is contained in a cross or the union of a line and a point. 
\end{enumerate}

Apply Lemma~\ref{lem:F'-vare-C'-vare} with $\vare=1/m$ for all $m\in\bbn$.
Let $\hat F_m=\hat F_{1/m}$ and for every $g\in \hat F_m$ let $\hat C_{m,g}=\hat C_{{1}/{m},g}$ denote the sets obtained by that lemma.
Define
\[
\hat F=\bigcap_{\ell\geq 1}\bigcup_{m\geq\ell} \hat F_m.
\]
Then $\hat F$ is conull in $F$.

Moreover, for every $g\in \hat F$ there exists a subsequence $m_k$ so that $g\in \hat F_{m_k}$
for all $k$. Let $\hat C_g=\bigcap_{\ell\geq 1}\bigcup_{k\geq\ell} \hat C_{m_k,g}$. Then $\hat C_g\subset C_g$ is conull in $C_g$.
Moreover, $\hat C_g$ satisfies (1) or~(2) in the claim --- recall that the same property holds for $\hat C_{m_k,g}$ by Lemma~\ref{lem:F'-vare-C'-vare}.

We now show that~(1) above holds almost surely. To see this, 
set $\mathcal L_1:=\{\tc_h:h\in\PGL_2(\lf_v)\}$ and $\mathcal L_2:=\{$union of a line and a point or crosses$\}$. 
Note that both $\mathcal L_1$ and $\mathcal L_2$ are $\Gamma$-invariant. 
Moreover, $\Psi$ is $\Gamma$-equivariant and $\Gamma$ acts ergodically on $\mathcal C$. 
Therefore, either the essential image of $\Psi|_C$ belongs to $\mathcal L_1$ a.e.\ $C\in\mathcal C$ or the essential image of $\Psi|_C$ belongs to $\mathcal L_2$ a.e.\ $C\in\mathcal C$.

Let $\xi\neq \xi'\in \mathbb S^2$ be so that $\Psi(\xi)=([r,1],[s,1])$ and $\Psi(\xi')=([r',1],[s',1])$ 
with $r\neq r'$ and $s\neq s'$ --- in view of Lemma~\ref{lem:PsiGmZd} we may find such points. 
Then there are exactly two crosses passing through both of $\Psi(\xi)$ and $\Psi(\xi')$; similarly for union of a line and a point. 
However, the set of circles in $\mathbb S^2$ passing through $\{\xi,\xi'\}$ covers the entire $\mathbb S^2$. 
Therefore, the essential image of $\Psi|_C$ belongs to $\mathcal L_2$ a.e.\ $C\in\mathcal C$ would contradict Lemma~\ref{lem:PsiGmZd}.

In consequence, we have: for a.e.\ $g\in \hat F$ there exists some $h_g\in\PGL_2(\lf_v)$ so that $\Psi(\hat C_g)\subset\tc_{h_g}$.
Since $\Psi$ is $\Gamma$-equivariant, this concludes the proof of the proposition.   
\end{proof}

\bibliographystyle{amsplain}
\bibliography{papers}

\providecommand{\bysame}{\leavevmode\hbox to3em{\hrulefill}\thinspace}
\providecommand{\MR}{\relax\ifhmode\unskip\space\fi MR }
\providecommand{\MRhref}[2]{%
  \href{http://www.ams.org/mathscinet-getitem?mr=#1}{#2}
}
\providecommand{\href}[2]{#2}
\begin{thebibliography}{10}

\bibitem{BFMS}
U.~Bader, D.~Fisher, N.~Miller, and M.~Stover, \emph{Arithmeticity,
  superrigidity, and totally geodesic submanifolds}, arXiv:1903.08467 (2019).

\bibitem{BQ-III}
Yves Benoist and Jean-Francois Quint, \emph{Stationary measures and invariant
  subsets of homogeneous spaces (iii)}, Annals of Mathematics \textbf{178}
  (2013), 1017--1059.

\bibitem{Berger-CLT}
Marc~A. Berger, \emph{Central limit theorem for products of random matrices},
  Transactions of the American Mathematical Society \textbf{285} (1984), no.~2,
  777--803.

\bibitem{Chow}
Y.~S. Chow, \emph{A martingale inequality and the law of large numbers}, Proc.
  Amer. Math. Soc. \textbf{11} (1960), 107--111. \MR{0112190}

\bibitem{Corl}
Kevin Corlette, \emph{Archimedean superrigidity and hyperbolic geometry},
  Annals of Mathematics \textbf{135} (1992), no.~1, 165--182.

\bibitem{DHM-Prob}
K.~Delp, D.~Hoffoss, and Manning J.F., \emph{Problems in groups, geometry, and
  three-manifolds}, arXiv:1512.04620. (2015).

\bibitem{EMM-Upp}
Alex Eskin, Gregory Margulis, and Shahar Mozes, \emph{Upper bounds and
  asymptotics in a quantitative version of the oppenheim conjecture}, Annals of
  Mathematics \textbf{147} (1998), no.~1, 93--141.

\bibitem{Fa-Kl}
I.~Fazekas and O.~Klesov, \emph{A general approach to the strong laws of large
  numbers}, Teor. Veroyatnost. i Primenen. \textbf{45} (2000), no.~3, 568--583.
  \MR{1967791}

\bibitem{FLMS}
D.~Fisher, J.F. Lafont, N.~Miller, and M.~Stover, \emph{Finiteness of maximal
  geodesic submanifolds in hyperbolic hybrids}, arXiv:1802.04619 (2018).

\bibitem{Furst-Kest}
H.~Furstenberg and H.~Kesten, \emph{Products of random matrices}, Ann. Math.
  Statist. \textbf{31} (1960), no.~2, 457--469.

\bibitem{Furst-Pois}
Harry Furstenberg, \emph{A poisson formula for semi-simple lie groups}, Annals
  of Mathematics \textbf{77} (1963), no.~2, 335--386.

\bibitem{Fu-Boundary}
Harry Furstenberg, \emph{Boundary theory and stochastic processes on
  homogeneous spaces}, Harmonic analysis on homogeneous spaces (Proc. Sympos.
  Pure Math., Vol. XXVI, Williams Coll., Williamstown, Mass., 1972), Amer.
  Math. Soc., Providence, R.I., 1973, pp.~193--229.

\bibitem{GarRag-FD}
H.~Garland and M.~S. Raghunathan, \emph{Fundamental domains for lattices in
  (r-)rank 1 semisimple lie groups}, Annals of Mathematics \textbf{92} (1970),
  no.~2, 279--326.

\bibitem{Golds-Marg}
I.~Ya. Goldsheid and G.~A. Margulis, \emph{Lyapunov exponents of a product of
  random matrices}, Uspekhi Mat. Nauk \textbf{44} (1989), no.~5(269), 13--60.
  \MR{1040268}

\bibitem{Grom-Sch}
Mikhail Gromov and Richard Schoen, \emph{Harmonic maps into singular spaces
  andp-adic superrigidity for lattices in groups of rank one}, Publications
  Math{\'e}matiques de l'Institut des Hautes {\'E}tudes Scientifiques
  \textbf{76} (1992), no.~1, 165--246.

\bibitem{Guivarch}
A.~Guivarc'h, \emph{Quelques proprietes asymptotiques des produits de matrices
  aleatoires}, Ecole d'Et{\'e} de Probabilit{\'e}s de Saint-Flour VIII-1978
  (Berlin, Heidelberg) (P.~L. Hennequin, ed.), Springer Berlin Heidelberg,
  1980, pp.~177--250.

\bibitem{Haj-Ren}
J.~H\'{a}jek and A.~R\'{e}nyi, \emph{Generalization of an inequality of
  {K}olmogorov}, Acta Math. Acad. Sci. Hungar. \textbf{6} (1955), 281--283.
  \MR{0076207}

\bibitem{Loeve}
Michel Lo\`eve, \emph{Probability theory. {F}oundations. {R}andom sequences},
  D. Van Nostrand Company, Inc., Toronto-New York-London, 1955. \MR{0066573}

\bibitem{MaclachlanReid}
Colin Maclachlan and Alan~W. Reid, \emph{The arithmetic of hyperbolic
  3-manifolds}, Graduate Texts in Mathematics, vol. 219, Springer-Verlag, New
  York, 2003.

\bibitem{MM}
G.~Margulis and A.~Mohammadi, \emph{Arithmeticity of hyperbolic 3-manifolds
  containing infinitely many totally geodesic surfaces}, arXiv:1902.07267v1
  (2019).

\bibitem{Margulis-Arith-Inv}
G.~A. Margulis, \emph{Arithmeticity of the irreducible lattices in the
  semisimple groups of rank greater than {$1$}}, Invent. Math. \textbf{76}
  (1984), no.~1, 93--120. \MR{739627}

\bibitem{Margulis-Book}
\bysame, \emph{Discrete subgroups of semisimple {L}ie groups}, Ergebnisse der
  Mathematik und ihrer Grenzgebiete (3) [Results in Mathematics and Related
  Areas (3)], vol.~17, Springer-Verlag, Berlin, 1991. \MR{1090825 (92h:22021)}

\bibitem{McR-Re}
D.~B. McReynolds and A.~W. Reid, \emph{The genus spectrum of hyperbolic
  3-manifolds}, Math. Research Letters \textbf{21} (2014), 169--185.

\bibitem{Mozes-Shah}
Shahar Mozes and Nimish Shah, \emph{On the space of ergodic invariant measures
  of unipotent flows}, Ergodic Theory Dynam. Systems \textbf{15} (1995), no.~1,
  149--159. \MR{1314973}

\bibitem{Reid-Commens}
Alan Reid, \emph{Totally geodesic surfaces in hyperbolic 3-manifolds},
  Proceedings of the Edinburgh Mathematical Society \textbf{34} (1991), 77 --
  88.

\bibitem{Selb}
Atle Selberg, \emph{On discontinuous groups in higher-dimensional symmetric
  spaces}, Contributions to function theory (internat. {C}olloq. {F}unction
  {T}heory, {B}ombay, 1960), Tata Institute of Fundamental Research, Bombay,
  1960, pp.~147--164. \MR{0130324}

\bibitem{Weil-1}
Andre Weil, \emph{On discrete subgroups of lie groups}, Annals of Mathematics
  \textbf{72} (1960), no.~2, 369--384.

\bibitem{Weil-2}
\bysame, \emph{Remarks on the cohomology of groups}, Annals of Mathematics
  \textbf{80} (1964), no.~1, 149--157.

\end{thebibliography}

\end{document}